\documentclass[11pt]{amsart}

\usepackage[margin=1.0in]{geometry}
\usepackage{amssymb,amsmath,amsthm,enumitem,graphicx,tikz}

\usepackage{hyperref}
\usepackage[noabbrev,capitalize]{cleveref}

\crefname{equation}{}{}
\usepackage{dsfont, color}
\usepackage{subcaption}
\usepackage[parfill]{parskip}
\usepackage{todonotes}
\setlist[itemize]{parsep=0pt}
\setlist[enumerate]{parsep=0pt}
\usepackage[ruled,vlined,linesnumbered,plain]{algorithm2e}
\makeatletter
\def\@algocf@pre@ruled{}
\def\@algocf@capt@ruled{}
\def\@algocf@post@ruled{}
\makeatother

\usepackage{etoolbox}
\makeatletter
\patchcmd{\@maketitle}
  {\ifx\@empty\@dedicatory}
  {\ifx\@empty\@date \else {\vskip3ex \centering\footnotesize\@date\par\vskip1ex}\fi
   \ifx\@empty\@dedicatory}
  {}{}
\patchcmd{\@adminfootnotes}
  {\ifx\@empty\@date\else \@footnotetext{\@setdate}\fi}
  {}{}{}
\makeatother
\calclayout
\makeatletter
\newtheorem*{rep@theorem}{\rep@title}
\newcommand{\newreptheorem}[2]{%
\newenvironment{rep#1}[1]{%
 \def\rep@title{#2 \ref{##1}}%
 \begin{rep@theorem}}%
 {\end{rep@theorem}}}
\makeatother

\newtheorem{theorem}{Theorem}[section] 
\newtheorem{proposition}[theorem]{Proposition}
\newtheorem{lemma}[theorem]{Lemma} 

\newtheorem{corollary}[theorem]{Corollary}
\newreptheorem{theorem}{Theorem}

\theoremstyle{definition}
\newtheorem{definition}[theorem]{Definition}

\newtheorem{question}[theorem]{Question}
\newreptheorem{question}{Question}

\newtheorem{example}[theorem]{Example}

\newtheorem{remark}[theorem]{Remark}
\newtheorem{observation}[theorem]{Observation}

\newcommand{\floor}[1]{\left\lfloor #1 \right\rfloor}
\newcommand{\ceil}[1]{\left\lceil #1 \right\rceil}
\newcommand{\abs}[1]{\left\lvert#1\right\rvert}

\newcommand{\paren}[1]{\left( #1 \right)}

\newcommand{\set}[1]{\left\{ #1 \right\}}

\newcommand{\wh}{\widehat}

\DeclareMathOperator{\diam}{diam}
\DeclareMathOperator{\supp}{supp}
\DeclareMathOperator{\adim}{adim}
\DeclareMathOperator{\bdim}{bdim}
\DeclareMathOperator*{\argmax}{\arg\!\max}
\DeclareMathOperator*{\argmin}{\arg\!\min}

\title{On the Broadcast Dimension of a Graph}
\author{Emily Zhang}

\begin{document}
\begin{abstract}
A function $f:V(G)\rightarrow \mathds{Z}^+ \cup \{0\}$ is a \textit{resolving broadcast} of a graph $G$ if, for any distinct $x,y\in V(G)$, there exists a vertex $z\in V(G)$ with $f(z)>0$ such that $\min\left\{d(x,z),f(z)+1\right\} \neq \min\left\{d(y,z),f(z)+1\right\}$. The \textit{broadcast dimension} of $G$ is the minimum of $\sum_{v\in V(G)}f(v)$ over all resolving broadcasts $f$ of $G$. 
The concept of broadcast dimension was introduced by Geneson and Yi as a variant of metric dimension and has applications in areas such as network discovery and robot navigation.

In this paper, we derive an asymptotically tight lower bound on the broadcast dimension of an acyclic graph in the number of vertices, and we show that a lower bound by Geneson and Yi on the broadcast dimension of a general graph in the adjacency dimension is asymptotically tight.
We also study the change in the broadcast dimension of a graph under a single edge deletion. We show that both the additive increase and decrease of the broadcast dimension of a graph under edge deletion is unbounded. Moreover, we show that under edge deletion, the broadcast dimension of any graph increases by a multiplicative factor of at most 3. These results fully answer three questions asked by Geneson and Yi. 
\end{abstract}

\maketitle

\section{Introduction}
Let $G=(V(G),E(G))$ be a finite, simple, and undirected graph of order $|V(G)|$. The distance $d_G(u,v)$ between two vertices $u,v\in V(G)$ is the length of the shortest path in the graph $G$ between $u$ and $v$ if they belong to the same connected component of $G$ and infinity otherwise. We omit the subscript $G$ if it is clear from the context. For a positive integer $k$ and vertices $u,v\in V(G)$, we define $d_k(u,v) := \min\{d(u,v),k+1\}$.

A set $S \subseteq V(G)$ is a \textit{resolving set} of $G$ if, for any distinct $x,y\in V(G)$, there is a vertex $z\in S$ such that $d(x,z)\neq d(y,z)$. Intuitively, a resolving set of $G$ is a set of landmark vertices, such that each vertex in $V(G)$ is uniquely characterized by its distances to the landmarks.
The \textit{metric dimension} $\dim(G)$ of $G$ is the cardinality of a smallest resolving set of $G$. 

Metric dimension was introduced by Slater \cite{slater1975leaves} in 1975, in connection with the problem of uniquely determining the location of an intruder in a network. Harary and Melter independently introduced the same concept in \cite{melter1976metric}. 
Metric dimension has since been heavily studied \cite{bailey2011base, caceres2007metric, chappell2008bounds, chartrand2003theory} and has applications in diverse areas such as chemistry \cite{chartrand2000resolvability}, pattern recognition and image processing \cite{melter1984metric}, and strategies for the Mastermind game \cite{chvatal1983mastermind}.
Khuller et al. \cite{khuller1996landmarks} considered robot navigation as another possible application of metric dimension. 
In that sense, a robot moving around in a space modeled by a graph can determine its distance to landmarks located at some of the vertices. The minimum number of landmarks required for the robot to uniquely determine its location on the graph is the metric dimension of the graph.




A set $A \subseteq V(G)$ is an \textit{adjacency resolving set} of $G$ if, for any distinct $x,y\in V(G)$, there is a vertex $z\in A$ such that $d_1(x,z) \neq d_1(y,z).$ The \textit{adjacency dimension} $\adim(G)$ of $G$ is the cardinality of a smallest adjacency revolving set of $G$. The concepts of adjacency resolving set and adjacency dimension were introduced by Jannesari and Omoomi \cite{jannesari2012metric} in 2012 as a tool for studying the metric dimension of lexicographic product graphs. 
The authors of \cite{jannesari2012metric} also considered robot navigation as a possible application of adjacency dimension: the minimum number of landmarks required for a robot moving from node to node on a graph to determine its location from only the landmarks adjacent to it is the adjacency dimension of the graph. 

A function $f:V(G)\rightarrow \mathds{Z}^+ \cup \{0\}$ is a \textit{resolving broadcast} of $G$ if, for any distinct $x,y\in V(G)$, there is a vertex $z\in \supp_G(f):= \{v\in V(G):f(v)>0\}$ such that $d_{f(z)}(x,z) \neq d_{f(z)}(y,z)$. The \textit{broadcast dimension} $\bdim(G)$ of $G$ is the minimum of $c_f(G) := \sum_{v\in V(G)}f(v)$ over all resolving broadcasts $f$ of $G$. The concepts of resolving broadcast and broadcast dimension were introduced in 2020 by Geneson and Yi \cite{geneson2020broadcast}, who noted that broadcast dimension also has applications in robot navigation. In that sense, transmitters with varying range are located at some of the vertices of a graph.  A transmitter with range $k$ has cost $k$ for $k\in \mathds{Z}^+\cup \set{0}$. A robot moving around on the graph learns its distance to each transmitter that it is within range of and learns that it is out of range of the others. The minimum total cost of transmitters required for a robot to determine its location on the graph is the broadcast dimension.


We say that a resolving set, adjacency resolving set, or resolving broadcast of $G$ is \textit{efficient} if it achieves $\dim(G)$, $\adim(G)$, or $\bdim(G)$, respectively. 

\begin{example}
The following tree $T$ has different metric, adjacency, and broadcast dimension.
\begin{figure}[h]
\centering
\begin{subfigure}{.33\textwidth}
  \centering
\includegraphics[scale=.5]{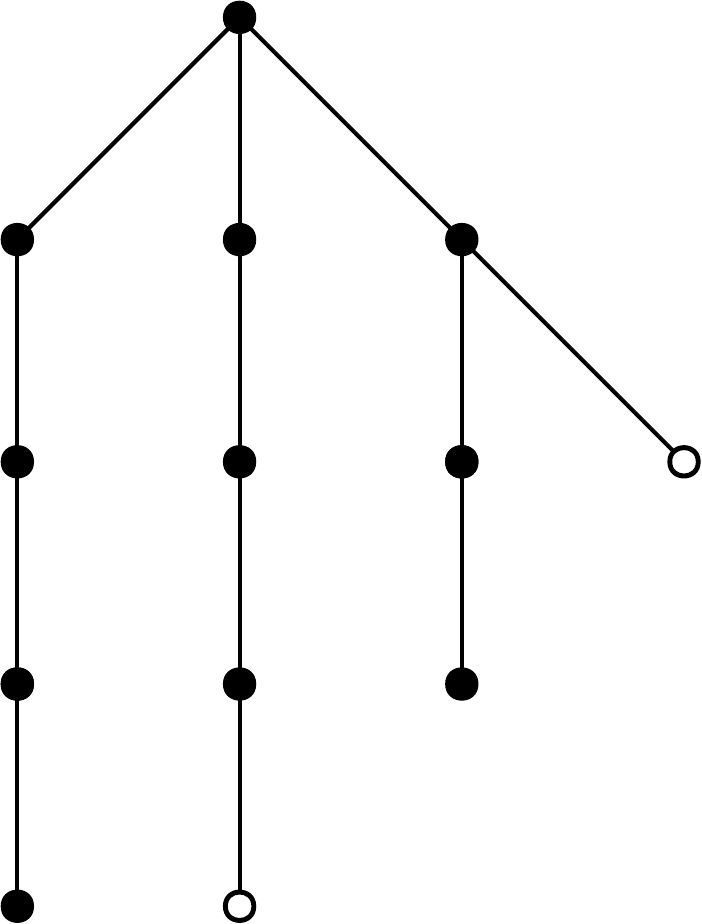}
\end{subfigure}%
\begin{subfigure}{.33\textwidth}
  \centering
\includegraphics[scale=.5]{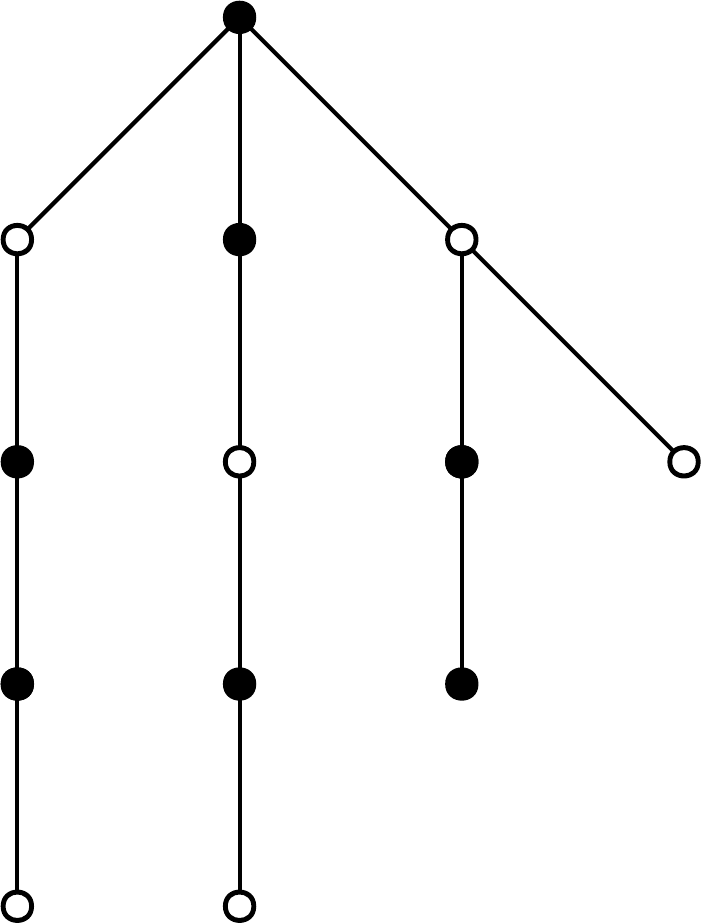}
\end{subfigure}
\begin{subfigure}{.33\textwidth}
  \centering
\includegraphics[scale=.5]{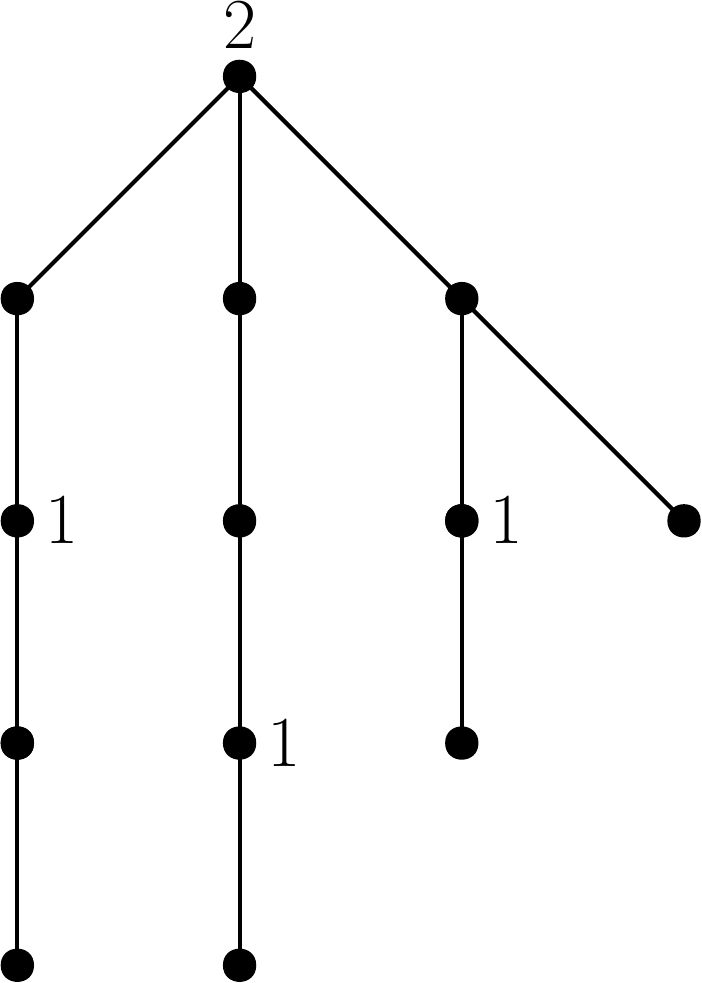}
\end{subfigure}

\begin{minipage}{.33\textwidth}
  \centering
  \vspace*{.5cm}
  $\dim(T) = 2$
\end{minipage}%
\begin{minipage}{.33\textwidth}
  \centering
    \vspace*{.5cm}
  $\adim(T) = 6$
\end{minipage}
\begin{minipage}{.33\textwidth}
  \centering
    \vspace*{.5cm}
  $\bdim(T) = 5$
\end{minipage}
\caption{Three copies of tree $T$. An efficient resolving set is shown with open circles in the first copy; an efficient adjacency resolving set is shown with open circles in the second copy; an efficient resolving broadcast is labeled on the third copy.}
 \label{figure: example}
\end{figure}
\end{example}


In \cite{geneson2020broadcast}, Geneson and Yi proved an asymptotic lower bound of $\Omega(\log{n})$ on the adjacency and broadcast dimension of graphs of order $n$, and they further demonstrated that this lower bound is asymptotically tight using a family of graphs from \cite{zubrilina2018edge}. 

\begin{theorem} [\cite{geneson2020broadcast}]
\label{thm: lowerbound_adim_bdim}
For all graphs $G$ of order $n$, we have $$n\geq \adim(G)\geq \bdim(G) = \Omega(\log{n}).$$
\end{theorem}

We improve the lower bound on the broadcast dimension of acyclic graphs of order $n$ from $\Omega(\log{n})$ to $\Omega(\sqrt{n})$ and show that this improved lower bound is asymptotically tight.

\begin{theorem} \label{thm: asymptotic_lower_bound1}
For all acyclic graphs $G$ of order $n$, we have ${\bdim}(G) =  \Omega(\sqrt{n})$, and this lower bound is asymptotically optimal.
\end{theorem}
Since the broadcast dimension is a generalization of the adjacency dimension, a natural question is how these quantities relate. 
\cref{thm: lowerbound_adim_bdim} gives that $\bdim(G) = \Omega\paren{\log\paren{\adim(G)}}$. In the following question, Geneson and Yi ask whether or not this lower bound is asymptotically optimal. 

\begin{question} (\cite{geneson2020broadcast})\textbf{.} \label{Question: 1}
Is there a family of graphs $\set{G_k}_{k\in \mathds{Z}^+}$ with $\bdim(G_k)=\Theta(k)$ and $\adim(G_k) = 2^{\Omega(k)}$ for every $k\in \mathds{Z}^+$?
\end{question}

We resolve \cref{Question: 1} affirmatively by constructing such a family of graphs. 
Thus, we complete the characterization of how the broadcast dimension of a graph $G$ can vary in the adjacency dimension of $G$: $\adim(G) \geq \bdim(G)  = \Omega( \log(\adim(G)))$, where both sides are tight. Our construction directly implies the following theorem.

\begin{theorem} \label{thm: asymptotica_lower_bound2}
The lower bound $\bdim(G) = \Omega\paren{\log\paren{\adim(G)}}$ is asymptotically optimal.
\end{theorem}

The question of the effect of vertex or edge deletion on the metric dimension of a graph was raised by Chartrand and Zhang in \cite{chartrand2003theory} as a fundamental question in graph theory. 
In \cite{geneson2020broadcast}, Geneson and Yi studied the effect of vertex deletion on the broadcast dimension of a graph, and they ask two corresponding questions for edge deletion.

\begin{question} 
(\cite{geneson2020broadcast})\textbf{.} 
\label{Question: 2}
 Is there a family of graphs $\set{G_k}_{k\in \mathds{Z}^+}$ such that $\bdim(G_k)-\bdim(G_k-e_k)$ can be arbitrarily large, where $e_k\in E(G_k)$?
\end{question}

\begin{question} 
(\cite{geneson2020broadcast})\textbf{.} 
\label{Question: 3}
For any graph $G$ and any $e=uv\in E(G)$, is it true that $\bdim(G-e) - \bdim(G) \leq d_{G-e}(u,v) - 1$?
\end{question}

Let $e=uv$ denote an edge of a connected graph $G$ such that $G-e$ is also a connected graph.
We resolve the first question affirmatively and show that the bound proposed in the second question can fail. In fact, the value $\bdim(G-e) - \bdim(G)$ can be arbitrarily larger than $d_{G-e}(u,v)$. We also show that while $\bdim(G-e) - \bdim(G)$ can be arbitrarily large, the ratio $\frac{\bdim(G-e)}{\bdim(G)}$ is bounded from above.

\begin{theorem}
\label{thm: edge_deletion1}
The value $\bdim(G) - \bdim(G-e)$ can be arbitrarily large. 
\end{theorem}

\begin{theorem}
\label{thm: edge_deletion2}
The value $\bdim(G-e) - \bdim(G)$ can be arbitrarily larger than $d_{G-e}(u,v)$.
\end{theorem}

\begin{theorem}
\label{thm: last}
For all graphs $G$ and any edge $e\in E(G)$, we have $\frac{\bdim(G-e)}{\bdim(G)} \leq 3$.
\end{theorem}

The rest of this paper is structured as follows. 
In \cref{Section: General}, we introduce relevant terminology and notation, and we record preliminary results on the metric, adjacency, and broadcast dimension of graphs that are necessary for the rest of the paper. 
In \cref{section: paths_and_cycles}, we examine the broadcast dimension of paths and cycles.
In \cref{Section: Acyclic}, we discuss results on the broadcast dimension of acyclic graphs and prove \cref{thm: asymptotic_lower_bound1}.
In \cref{Section: Comparing}, we resolve \cref{Question: 1} affirmatively and prove \cref{thm: asymptotica_lower_bound2}. In \cref{Section: Edge_Deletion}, we prove Theorems \ref{thm: edge_deletion1}, \ref{thm: edge_deletion2}, and \ref{thm: last}.
Finally in \cref{Section: Future_Work}, we conclude with some open problems about broadcast dimension.

\section{Preliminaries}
\label{Section: General}
In this section, we first introduce relevant terminology and notation that we will use throughout the paper. We then record some preliminary results on the metric, adjacency, and broadcast dimension of graphs. For the rest of this section, we let $f:V(G)\rightarrow \mathds{Z}^+ \cup \{0\}$ for graph $G=(V(G), E(G))$.

We denote by $P_n$, $C_n$, and $K_n$ the path, cycle, and complete graph on $n$ vertices, respectively. We say $\diam(G) = \max\set{d(u,v)\mid u,v\in V(G)}$. 
We denote by $\mathbf{1}$ the vector with 1 for each entry and $\mathbf{2}$ the vector with 2 for each entry, where the length of the vector is inferred from context. For an arbitrary set $S$, a totally ordered set $Y$, and a function $g:S\rightarrow Y$, we define $\argmax_{x\in S} g(x)$ to be any $x^*\in S$ such that $g(x) \leq g(x^*)$ for all $x\in S$. We define $\argmin_{x\in S} g(x)$ analogously.

\begin{definition}
A vertex $z\in \supp_G(f)$ \textit{resolves} a pair of distinct vertices $x,y\in V(G)$ if $$d_{f(z)}(x,z) \neq d_{f(z)}(y,z).$$
\end{definition}

In order for a vertex $z\in \supp_G(f)$ to resolve a pair of vertices $x,y\in V(G)$, we must have $f(z)\geq d(x,z)$ or $f(z)\geq d(y,z)$. We formally define this notion below.

\begin{definition}
A vertex $z\in \supp_G(f)$ \textit{reaches} a vertex $v\in V(G)$ with respect to $f$ if $f(z) \geq d(v,z)$, and the function $f$ \textit{reaches} a vertex $v\in V(G)$ if there is a vertex $z\in \supp_G(f)$ that reaches $v$. 
\end{definition}

By definition, the function $f$ is a resolving broadcast of $G$ if and only if every pair of distinct vertices in $V(G)$ is resolved by a vertex in $\supp_G(f)$. Thus, any resolving broadcast $f$ of $G$ must reach all but at most one vertex in $V(G)$. 
Equivalently, the function $f$ is a resolving broadcast of $G$ if and only if every vertex of $G$ is \textit{distinguished}; that is, every vertex of $G$ is uniquely characterized by its distances to the vertices in $\supp_G(f)$ that reach it. We formally define this term below.

\begin{definition}
Let $k = \abs{\supp_G(f)}$.
The \textit{broadcast representation} of a vertex $v\in V(G)$ with respect to $f$ is the $k$-vector $b_f(v) =\allowbreak \paren{d_{f(u_1)}(v, u_1), \dots, d_{f(u_k)}(v, u_k)}$ for $u_i \in \supp_G(f)$.
We say that a vertex $v\in V(G)$ is \textit{distinguished}  if it has a unique broadcast representation $b_f(v)$. 
\end{definition}

The following observations give insight into how the metric, adjacency dimension, and broadcast dimension of graphs are related and will be useful throughout the rest of the paper.

\begin{observation} (\cite{geneson2020broadcast})\textbf{.}
The following properties hold for any graph $G$.
\begin{enumerate}
	\item We have $\dim(G) \leq \bdim(G) \leq \adim(G)$.
	\item If $\diam(G)\leq 2$, then we have $\dim(G)=\bdim(G)=\adim(G)$.
\end{enumerate}
\end{observation}

The \textit{closed neighborhood} of a vertex $v\in V(G)$ is $N[v] = \{u\in V(G): uv\in E(G)\} \cup \{v\}$. 
Two distinct vertices $u,v \in V(G)$ are called \textit{twin vertices} if $N[u] =N[v]$.

\begin{observation}
\label{obs: twin}
If $u,v\in V(G)$ are twin vertices, then the following properties hold.
\begin{enumerate}
\item For any resolving set $S$ of $G$, we have that $u\in S$ or $v\in S$ \cite{hernando2010extremal}.
\item For any adjacency resolving set $A$ of $G$, we have that $u\in A$ or $v\in A$ \cite{jannesari2012metric}.
\item For any resolving broadcast $f$ of $G$, we have that $u\in \supp_G(f)$ or $v\in \supp_G(f)$ \cite{geneson2020broadcast}.
\end{enumerate}
\end{observation}

\section{Paths and Cycles}
\label{section: paths_and_cycles}
Here we restrict our attention to path and cycle graphs.
It is easy to see that $\dim(P_n) = 1$ and $\dim(C_n) = 2$ for every integer $n\geq 3$. The adjacency dimension and the broadcast dimension, respectively, of paths and cycles were determined in \cite{jannesari2012metric} and \cite{geneson2020broadcast}.

\begin{theorem} [\cite{jannesari2012metric}] \label{thm: paths_adj}
For every integer $n\geq 4$, we have $\adim(P_n) = \adim(C_n) = \floor{\frac{2n+2}{5}}$.
\end{theorem}

\begin{theorem} [\cite{geneson2020broadcast}]
\label{thm: paths}
For every integer $n\geq 4$, we have $\bdim(P_n) = \bdim(C_n) = \floor{\frac{2n+2}{5}}$.
\end{theorem}

In this section, we prove the following result on efficient resolving broadcasts of paths and cycles.

\begin{proposition}
\label{prop: my_paths_cycles}
For every $n \in \mathds{Z}^+$ and $G\in \set{P_n,C_n}$, if $f$ is an efficient resolving broadcast of $G$, then $f(v) \leq 2$ for all $v\in V(G)$.
\end{proposition}

We begin with two lemmas.
In the proof of \cref{thm: paths}, Geneson and Yi proved the following useful fact, which we state here as a lemma. We include the proof for completeness.

\begin{lemma} [\cite{geneson2020broadcast}]
\label{lemma: at_most_one}
For every $n \in \mathds{Z}^+$ and every efficient resolving broadcast $f$ of $G\in \set{P_n,C_n}$, there is an efficient resolving broadcast $f'$ of $G$
with the following properties. 
\begin{enumerate}
\item Every vertex reached by $f$ is also reached by $f'$. 
\item For all $v\in V(G)$, we have $f'(v)\leq 1$.
\end{enumerate}
\end{lemma}

\begin{proof} Let $G$ be the path $v_1,\dots,v_{n}$ or the cycle $v_1,\dots,v_{n},v_1$.
Let $f_0$ be any efficient resolving broadcast of $G$. If $f_0(v)\leq 1$ for all $v\in V(G)$, then we are done. Otherwise, we repeatedly modify $f_i$ to obtain a new efficient resolving broadcast $f_{i+1}$ that satisfies the following monovariant: for integer $k$, let $U_k = \{v\in V(G): f_k(v)>1\} $ and $S_k = \sum_{v\in U_k} f_k(v)$, then $S_{i+1} < S_i$.

Let $v_j\in V(G)$ be any vertex with $x:=f_i(v_j)> 1$. If $v_j$ is a leaf and $x =2$, we set $f_{i+1}(v_j)=1$ and $f_{i+1}(u)= \max\{f_i(u), 1\}$, where $u$ is the vertex adjacent to $v_j$. 
Otherwise, we set $f_{i+1}(v_j) = x - 2$, and we let $u_1$ and $u_2$ be the vertices $v_{(j+x-1)\mod n}$ and $v_{(j-x+1)\mod n}$, respectively. 
We set $f_{i+1}(u_1)=\max\{f_i(u_1), 1\}$ and $f_{i+1}(u_2) =  \max\{f_i(u_2), 1\}$. 
The maximum value is used for vertices assigned multiple values for $f_{i+1}$, and $f_{i+1}(v) = f_i(v)$ for any vertex $v$ not assigned any value for $f_{i+1}$. 
This process will terminate after finitely many steps because of the monovariant on $S_i$, yielding a resolving broadcast that satisfies the description of $f'$. 
\end{proof}

The proof of \cref{lemma: hat_bdim} uses some ideas from observations made in \cite{buczkowski2003k} about the metric dimension of a wheel $W_n = C_n + K_1$ for integer $n\geq 3$. To state the lemma, we need the following definition.

\begin{definition}
For a graph $G$, the value $\wh{\bdim}(G)$ is the minimum of $\sum_{v\in V(G)} f(v)$ over all resolving broadcasts $f$ of $G$ such that every vertex $v \in V(G)$ is reached by at least one vertex $z\in \supp_G(f)$. This differs from $\bdim(G)$ because one vertex may be unreached by a resolving broadcast. 
\end{definition}

\begin{observation}
For all graphs $G$, we have $$\wh{\bdim}(G) = \bdim(G\cup K_1) \quad \text{and} \quad \bdim(G)\leq \wh{\bdim}(G) \leq \bdim(G) + 1.$$
\end{observation}

\begin{lemma}
\label{lemma: hat_bdim}
For every integer $n\geq 4$, we have $\wh{\bdim}(P_n) =\wh{\bdim}(C_n)= \floor{\frac{2n+3}{5}}$.
\end{lemma}
\begin{proof}
Let $G$ be the path $v_1,\dots,v_{n}$ or the cycle $v_1,\dots,v_{n},v_1$.
First, we will show that $\wh{\bdim}(G) = \bdim(G)$ for $n\not\equiv 1 \pmod 5$. Define $g:V(G)\rightarrow \mathds{Z}^+ \cup \{0\}$ as follows: $g(v_i)$ is 1 if $i\equiv 2\pmod 5$ or  $i\equiv 4\pmod 5$ and 0 otherwise. Note that $g$ is a resolving broadcast of $G$ that achieves $\bdim(G)$ given in \cref{thm: paths} and that $g$ reaches all of the vertices of $G$ when $n\not\equiv 1 \pmod 5$.

Now, we will show that $\wh{\bdim}(G)= \bdim(G)+1$ for $n\equiv 1 \pmod 5$. 
Let $n= 5x + 1$ for some positive integer $x$; then, we have $\bdim(G) = \floor{\frac{10x+4}{5}}=2x$.
It suffices to show that for any efficient resolving broadcast $f$ of $G$, there is a vertex not reached by $f$. 
By \cref{lemma: at_most_one}, there is an efficient resolving broadcast $f'$ of $G$ with $f'(v)\leq 1$ for all $v\in V(G)$ that reaches all of the vertices reached by $f$. For the sake of contradiction, we assume that $f'$ reaches all of the vertices, and so there is no vertex $v\in V(G)$ with $b_{f'}(v) = \mathbf{2}$.

A \textit{gap} of graph $G$ is a maximal connected subgraph of $G$ that only consists of vertices that are not in $\supp_G(f')$. If two gaps are adjacent to the same vertex in $\supp_G(f')$, then we call them \textit{neighboring gaps}.
No gap can contain three vertices, since the vertex in the middle of the gap would have broadcast representation $\mathbf{2}$. Additionally, any neighboring gap of a gap that contains two vertices must contain only one vertex, since otherwise there exists five consecutive vertices of $G$ where the vertex $m$ in the middle is the only one in $\supp_G(f')$, and the two vertices adjacent to $m$ would have the same broadcast representation.

If $G$ is $C_n$, then of the $\bdim(G)$ gaps, at most $\floor{\frac{\bdim(G)}{2}}$ gaps contain two vertices, and none contain three vertices. Thus, $n\leq  2\bdim(C_n) + \floor{\frac{\bdim(C_n)}{2}} =5x$. Similar reasoning yields $n\leq 5x$ if $G$ were instead $P_n$. Since $G$ is a graph of order $5x+1$, we have reached a contradiction.
\end{proof}

With the above lemma, we are now able to prove \cref{prop: my_paths_cycles}.

\begin{proof}[Proof of \cref{prop: my_paths_cycles}]
Let $G$ be the path $v_1,\dots,v_{n}$ or the cycle $v_1,\dots,v_{n},v_1$, and let $f$ be an efficient resolving broadcast of $G$.
If $n\leq 6$, then $\bdim(G) \leq 2$ by \cref{thm: paths}, so $f(v) \leq 2$ for all $v\in V(G)$. Thus, we consider $n\geq 7$.

Let $v_i = \argmax_{v\in V(G)}(f(v))$.
For the sake of contradiction, we assume that $f(v_i)\geq 3$. 
If vertex $v_i$ were a leaf (say $i=1$), then a function $g$ that is identical to $f$, except with $g(v_3)=f(v_1)-2$ and $g(v_1)= 1$, is a resolving broadcast of $G$ with 
$c_g(G) < c_f(G)$, contradicting the efficiency of $f$. Thus, $v_i$ has two neighbors. At least one of the neighbors of $v_i$ must be reached by some other vertex $v_j\neq v_i$ or else the two neighbors of $v_i$ would not be distinguished. 

First, we will show that $f$ is inefficient if $f(v_j)\geq 2$. Let $T$ be the set of vertices that are reached by $v_i$ or  $v_j$. Note that $|T|\leq 2f(v_i)+f(v_j)+2$. By \cref{lemma: hat_bdim}, the vertices in $T$ can be reached and distinguished with a total cost of $\floor{\frac{2|T|+3}{5}}$, which is less than $f(v_i) + f(v_j)$ when $f(v_i) \geq 3$ and $f(v_j)\geq 2$.

Thus, we must have $f(v_j)=1$, so
$|T|\leq 2f(v_i)+1$ since $v_j$ cannot reach any vertex that $v_i$ does not reach. 
By \cref{lemma: hat_bdim}, the vertices in $T$ can be reached and distinguished with a total cost of $$\floor{\frac{2|T|+3}{5}}\leq \frac{4f(v_i)+5}{5} < f(v_i)+1 = f(v_i)+f(v_j).$$
This contradicts the efficiency of resolving broadcast $f$. 
\end{proof}

\section{Results on Acyclic Graphs}
\label{Section: Acyclic}

In this section, we discuss some results on the broadcast dimension of acyclic graphs, and we prove \cref{thm: asymptotic_lower_bound1}. We make use of standard terminology for trees: a \textit{major vertex} in a tree $T$ is a vertex of degree at least three, and a \textit{leaf} of $T$ is a vertex of degree one. 

For any graph $G$, showing that a function $g:V(G)\rightarrow \mathds{Z}^+ \cup \{0\}$ is a resolving broadcast of $G$ gives an upper bound of $c_g(G)$ on $\bdim(G)$. 
On the other hand, obtaining a nice lower bound on $\bdim(G)$ is oftentimes less straightforward. 
The result on twin vertices from \cref{obs: twin} is a useful tool for lower bounding $\bdim(G)$. 
In this section, we use a different approach to derive a lower bound on the broadcast dimension of trees: we consider the number of unique broadcast representations of the vertices of a tree $T$ with respect to various functions $f:V(T)\rightarrow \mathds{Z}^+ \cup \{0\}$. This motivates the following definition.

\begin{definition}
For a graph $G$ of order $n$ and a function $f:V(G)\rightarrow \mathds{Z}^+ \cup \{0\}$, we say that $B_G(f)$ is the number of unique broadcast representations of the vertices of $G$ with respect to $f$. That is,
$$B_G(f) =\abs{\set{b_f(v) \mid v\in V(G)}}.$$ Note that $B_G(f)=n$ if and only if $f$ is a resolving broadcast of $G$.
\end{definition}




The following lemma will be useful in the proof of \cref{thm: acyclic_lower_bound}.
\begin{lemma}
\label{lemma: reachable_from_a}
Let $T$ be a tree with resolving broadcast $f$, and let $a, b, v, x\in V(T)$ such that the following inequalities hold:
\begin{align*}
& f(a)-d(a,x)\geq f(v)-d(v,x),\\
& f(b)-d(b,x) \geq f(v)-d(v,x), \\
& f(a)-d(a,v)\geq f(b)-d(b,v).
\end{align*}
Then every vertex of $T$ that is reached by both $b$ and $v$ is also reached by $a$.
\end{lemma}
\begin{proof}
We consider four possible orientations of the vertices $a$, $b$, and $v$ (see \cref{figure: Cases}). 

\noindent\textbf{Case 1.} There is not a path in $T$ through vertices  $a$, $b$, and $v$.\\
Let $c$ be the major vertex of $T$ such that the path from $c$ to $a$, the path from $c$ to $b$, and the path from $c$ to $v$ do not share any edges. 

In this case, $f(a)-d(a,v)\geq f(b)-d(b,v)$ implies that 
\begin{equation}
\label{eq:1}
f(a)-d(a,c)\geq f(b)-d(b,c).
\end{equation}

If the path from $x$ to $a$ does not go through $c$, then both the path from $x$ to $b$ and the path from $x$ to $v$ must pass through $c$, so $f(b)-d(b,x) \geq f(v)-d(v,x)$ implies that $f(b)-d(b,c) \geq f(v)-d(v,c)$. Combining this inequality with \eqref{eq:1}, we have
\begin{equation} \label{eq:2}
f(a)-d(a,c)\geq f(v)-d(v,c).
\end{equation}
Alternatively, if the path from $x$ to $a$ does go through $c$, then $f(a)-d(a,x)\geq f(v)-d(v,x)$ directly implies \eqref{eq:2}. 
Thus, the inequality in \eqref{eq:2} holds no matter where vertex $x$ is.

The inequality in \eqref{eq:1} shows that any vertex reached by $b$ with a path to $b$ that goes through $c$ is reached by $a$.
Similarly, the inequality in \eqref{eq:2} shows that
any vertex reached by $v$ with a path to $v$ that goes through $c$ is reached by $a$. 
Thus, any vertex that is reached by both  $b$ and $v$ is also reached by $a$.

\noindent\textbf{Case 2.} $d(a,v)+d(v,b) = d(a,b)$.\\
If the path from $x$ to $b$ does not go through $v$, then 
\begin{align*}
f(a)-d(a,x)\geq f(v)-d(v,x)
&\implies f(a)\geq d(a,v) + d(v,x)+ f(v)-d(v,x)\\
&\implies f(a)-d(a,v) \geq  f(v).
\end{align*}
Alternatively, if the path from $x$ to $b$ does go through $v$, then replacing $a$ with $b$ in the above inequalities, we get $f(b)-d(b,v) \geq  f(v)$, which implies that $f(a)-d(a,v)\geq f(v)$.

Thus, no matter where vertex $x$ is, we have $f(a)-d(a,v) \geq  f(v)$, which shows that $a$ reaches all of the vertices reached by $v$. 

\noindent\textbf{Case 3.} $d(b,a)+d(a,v) = d(b,v)$.

\noindent\textbf{Case 4.} $d(a,b)+d(b,v) = d(a,v)$.

It is easy to see that the lemma is true for Cases 3 and 4 by direct observation or by performing analysis similar to the analysis shown for Cases 1 and 2.
\end{proof}

\begin{theorem}
\label{thm: acyclic_lower_bound}
For all trees $T$ of order $n$, we have ${\bdim}(T) \geq  \sqrt{\frac{n}{6}}.$
\end{theorem}

\begin{proof}
Let $T$ be a tree of order $n$, and let $f$ be any resolving broadcast of $T$. We define $f': V(T)\rightarrow \mathds{Z}^+ \cup \{0\}$ such that $f'(v) = 0$ for all $v\in V(T)$. Note that $B_T(f') = 1$.
Let $x\in V(T)$ be any vertex. We order the vertices in $\supp_T(f)$ so that vertex $v\in \supp_T(f)$ comes before vertex $u\in \supp_T(f)$ in the ordering only if $f(v)-d(v,x)\geq f(u) -d(u,x)$.
We update the value of $f'(v)$ from 0 to $f(v)$ (notationally, $f'(v)\leftarrow f(v)$) one vertex $v\in \supp_T(f)$ at a time in the defined order until $f' = f$, and we consider the increase in $B_T(f')$ on each update.

\begin{figure}[t]
\centering
\begin{subfigure}{.5\textwidth}
  \centering
  \includegraphics[scale=.5]{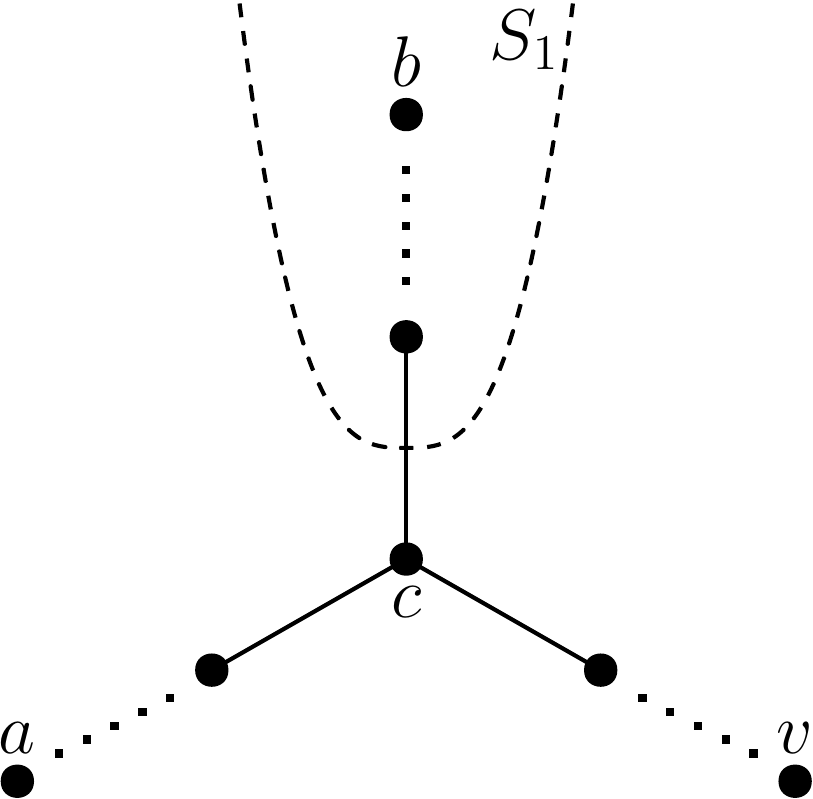}
\end{subfigure}%
\begin{subfigure}{.5\textwidth}
  \centering
  \includegraphics[scale=.5]{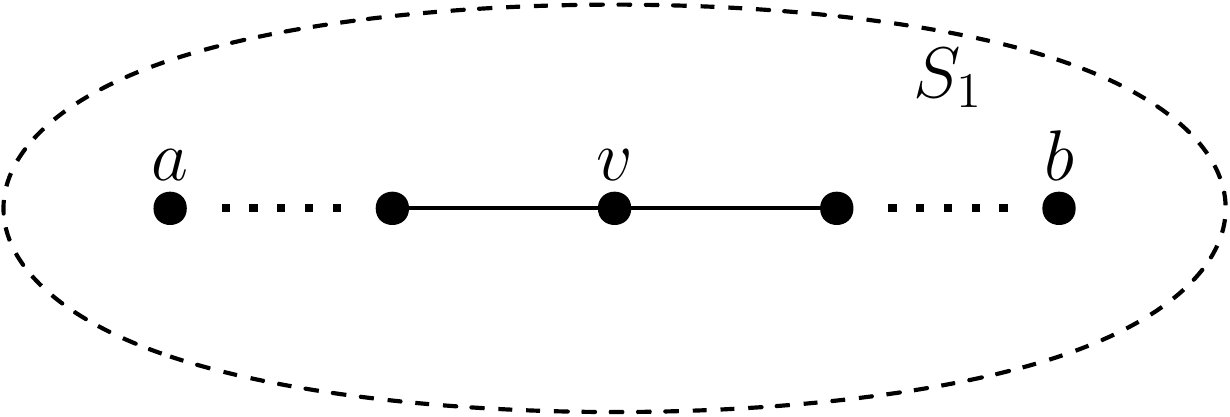}
\end{subfigure}
\begin{minipage}{.5\textwidth}
  \centering
  \vspace*{.3cm}
  Case 1
    \vspace*{.8cm}
\end{minipage}%
\begin{minipage}{.5\textwidth}
  \centering
    \vspace*{.3cm}
  Case 2
    \vspace*{.8cm}
\end{minipage}
\begin{subfigure}{.5\textwidth}
  \centering
  \includegraphics[scale=.5]{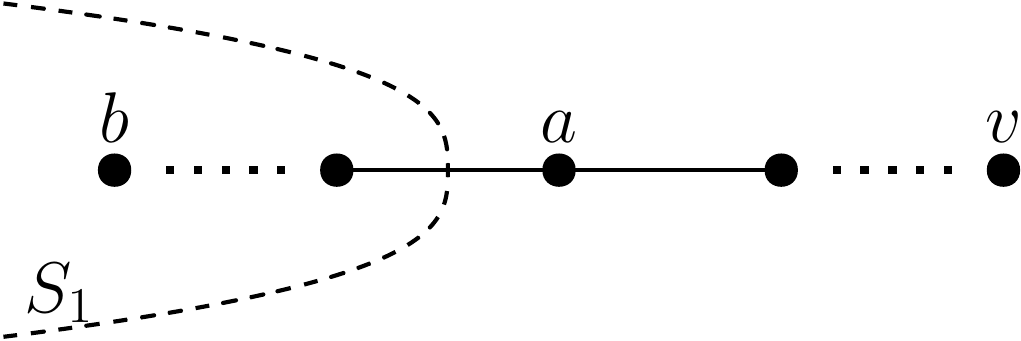}
\end{subfigure}%
\begin{subfigure}{.5\textwidth}
  \centering
  \includegraphics[scale=.5]{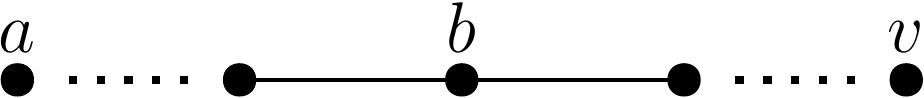}
\end{subfigure}
\begin{minipage}{.5\textwidth}
  \centering
  \vspace*{.3cm}
  Case 3
\end{minipage}%
\begin{minipage}{.5\textwidth}
  \centering
    \vspace*{.3cm}
  Case 4
\end{minipage}
\caption{The four cases from the proof of \cref{lemma: reachable_from_a} and the proof of \cref{thm: acyclic_lower_bound}. Note that all vertices may have larger degree than what is shown.
Any non-pictured vertex of the tree that is in $S$ (defined in the proof of \cref{thm: acyclic_lower_bound}) and adjacent to a vertex in $S_1$ is also in $S_1$.}
\label{figure: Cases}
\end{figure}

For a vertex $v\in \supp_T(f)$, let $W(v)$ be the set of vertices that can reach (with respect to $f'$) at least one vertex $u\in V(T)$ that is reached by $v$ (with respect to $f$). That is, 
\begin{equation*}
W(v) = \set{w \mid w\neq v, w\in \supp_T(f'), u\in V(T), f'(w)\geq d(u,w), f(v)\geq d(u,v)}.
\end{equation*}

If $W(v) = \emptyset$, then updating $f'(v) \leftarrow f(v)$ increases $B_T(f')$ by at most $f(v) + 1$, which is upper bounded by $2\paren{f\paren{v}}^2$ since $f(v) \geq 1$. 

If $|W(v)| = 1$, then we can make the following observations about the broadcast representation $b_{f'}(u)$ of any vertex $u$ reached by $v$ after the update $f'(v) \leftarrow f(v)$. There are $f(v) + 1$ possible values for the entry of $b_{f'}(u)$ corresponding to vertex $v$ and $2f(v) + 1$ possible values for the entry of $b_{f'}(u)$ corresponding to the vertex in $W(v)$. The rest of the entries of $b_{f'}(u)$ must be the maximal possible value for that entry.
Thus, $B_T(f')$ increases by at most $\paren{f(v) + 1}\paren{2f(v) + 1}\leq 6\paren{f\paren{v}}^2$ in this case.

Now we consider $|W(v)| > 1$. Let $a=\argmax_{u\in W(v)}\paren{f'(u) - d(u,v)}$ and $b\in W(v)-\{a\}$.
Let $\delta \geq 0$ such that the update $f'(v)\leftarrow f(v)$ increases $B_T(f')$ by $\delta$. 

\noindent\textbf{Claim.}  
If $f'(b)$ were instead zero, then the update $f'(v)\leftarrow f(v)$ would still increase $B_T(f')$ by at least $\delta$. 
\begin{proof}[Proof of claim.]
Let $S$ be the set of vertices reached by both $b$ and $v$, and let $S_0 = V(T) - S$. 
We consider four possible orientations of the vertices $a$, $b$, and $v$ (see \cref{figure: Cases}), and we show that, in each case, the vertices in $S$ can be split into two (possibly empty) sets $S_1$ and $S_2$ such the three properties listed below are satisfied. 
Note that showing this proves the claim. 

\begin{minipage}{30pt}
$\text{ }$
\end{minipage}%
\begin{minipage}{.9\textwidth}
\begin{enumerate}
\item [Property 1.] Before updating $f'(v)\leftarrow f(v)$, every vertex in $S_1$ has a different broadcast representation from every vertex in $V(T) - S_1$.
\item[Property 2.] Updating $f'(v)\leftarrow f(v)$ does not increase $\abs{\set{b_{f'}(v) \mid v\in S_1}}$.
\item[Property 3.]  If $f'(b)$ were instead zero, updating $f'(v)\leftarrow f(v)$ would increase $\abs{\set{b_{f'}(v) \mid v\in S_2\cup S_0}}$ by at least $\delta$.
\end{enumerate}
\end{minipage}

Since we made the updates $f'(a)\leftarrow f(a)$ and $f'(b)\leftarrow f(b)$ before the update $f'(v)\leftarrow f(v)$, we have $f(a)-d(a,x)\geq f(v)-d(v,x)$ and $f(b)-d(b,x) \geq f(v)-d(v,x)$. Because of the way we chose vertex $a$, we have $f(a)-d(a,v)\geq f(b)-d(b,v)$.
Thus, by \cref{lemma: reachable_from_a}, vertex $a$ also reaches all of the vertices in $S$. 

Because every vertex in $S_0$ is not reached by $b$ or not reached by $v$, the increase in $\abs{\set{b_{f'}(v) \mid v\in S_0}}$ after updating $f'(v)\leftarrow f(v)$ would be at least the same if $f'(b)$ were instead zero. In all four cases, if $S_1=\emptyset$, Properties 1 and 2 are trivially satisfied, and if $S_2 = \emptyset$, Property 3 is trivially satisfied.

\noindent\textbf{Case 1.} There is not a path in $T$ through vertices  $a$, $b$, and $v$.\\
Let $c$ be the major vertex of $T$ such that the path from $c$ to $a$, the path from $c$ to $b$, and the path from $c$ to $v$ do not share any edges. Let $S_1$ be the set of vertices in $S$ with a path to $b$ that does not go through $c$, and $S_2 = S-S_1$. Let $u_1\in S_1$. All other vertices with distance $d(u_1,a)$ to $a$ and $d(u_1,b)$ to $b$ are also in $S_1$ (Property 1) and have the same distance $d(u_1,a) - d(a,c)+d(c,v)$ to vertex $v$ (Property 2). Let $u_2\in S_2$. All of the vertices in $S_2$ that are distance $d(u_2,a)$ to vertex $a$ and distance $d(u_2,v)$ to vertex $v$ have the same distance
to vertex $b$ (Property 3).

\noindent\textbf{Case 2.} $d(a,v)+d(v,b) = d(a,b)$.\\
Let $S_1=S$ and $S_2 = \emptyset$. Let $u_1\in S_1$. All other vertices that are distance $d(u_1,a)$ from vertex $a$ and distance $d(u_1,b)$ from vertex $b$ are also in $S_1$ (Property 1) and are the same distance from vertex $v$ (Property 2). Property 3 is trivially satisfied.

\noindent\textbf{Case 3.} $d(b,a)+d(a,v) = d(b,v)$.\\
Let $S_1$ be the set of vertices in $S$ with a path to $b$ that does not go through $a$, and $S_2=S-S_1$. Let $u_1\in S_1$. All other vertices with distance $d(u_1,a)$ to $a$ and $d(u_1,b)$ to $b$ are also in $S_1$ (Property 1), and they all have the same distance $d(u_1,a) + d(a,v)$ to vertex $v$ (Property 2).
Let $u_2\in S_2$. All of the vertices in $S_2$ that are distance $d(u_2,a)$ to vertex $a$ have the same distance $d(u_2,a) + d(a,b)$ to vertex $b$ (Property 3).

\noindent\textbf{Case 4.} $d(a,b)+d(b,v) = d(a,v)$.\\
Let $S_1=\emptyset$ and $S_2 = S$. Properties 1 and 2 are trivially satisfied.
Let $u_2\in S_2$. All of the vertices with distance $d(u_2,a)$ to vertex $a$ and distance $d(u_2,v)$ to vertex $v$ have the same distance 
to vertex $b$ (Property 3). 
\end{proof}
The claim implies that the change in $B_T(f')$ after updating $f'(v) \leftarrow f(v)$ when $|W(v)| > 1$ is upper bounded by the change in $B_T(f')$ after updating $f'(v) \leftarrow f(v)$ if we instead had $W(v)=\{a\}$. Thus, every update increases $B_T(f')$ by at most $6\paren{f(v)}^2$, and the very first update increases $B_T(f')$ by at most $2\paren{f(v)}^2$. 
Since we started out with $B_T(f') = 1$, and we must have $B_T(f') = n$ after finishing all of the updates, we have that $c_f(T) \geq \sqrt{\frac{n}{6}}$ for any resolving broadcast $f$ of $T$.
\end{proof}


Because the broadcast dimension of a disconnected graph is at least the sum of the broadcast dimensions of all of its connected components, \cref{thm: acyclic_lower_bound} directly implies the following corollaries.
\begin{corollary}
\label{corollary: acyclic}
For all acyclic graphs $G$ of order $n$, we have 
${\bdim}(G) =  \Omega(\sqrt{n})$.
\end{corollary}

\begin{corollary}
\label{corollary: adim_of_acyclic}
For all acyclic graphs $G$ of order $n$, we have $\adim(G) = \Omega(\sqrt{n})$.
\end{corollary}

\begin{corollary}
\label{corollary: adim_of_acyclic2}
For all acyclic graphs $G$ of order $n$, we have $\adim(G) = O\paren{\paren{\bdim(G)}^2}$.
\end{corollary}

Now we will show that the bound from \cref{thm: acyclic_lower_bound} is sharp up to a constant factor and that the asymptotic bounds from \cref{corollary: acyclic} and \cref{corollary: adim_of_acyclic2} are asymptotically optimal.
We do so by finding a family of trees that achieves these bounds up to a constant factor. This family of graphs will also be used to study edge deletion in \cref{Section: Edge_Deletion}.
\begin{definition}
\label{def: F_k}
For every $k\in \mathds{Z}^+ \cup \{0\}$, graph $L_k$ is the path $v_0, \dots, v_k$.
The graph $F_k$ is $L_k$ with a path $P_i$ connected to $v_i$ for each $1\leq i\leq k$. (See \cref{figure: F3} for the graph $F_3$.)
\end{definition}
\begin{figure}[h]
\includegraphics[scale=.5]{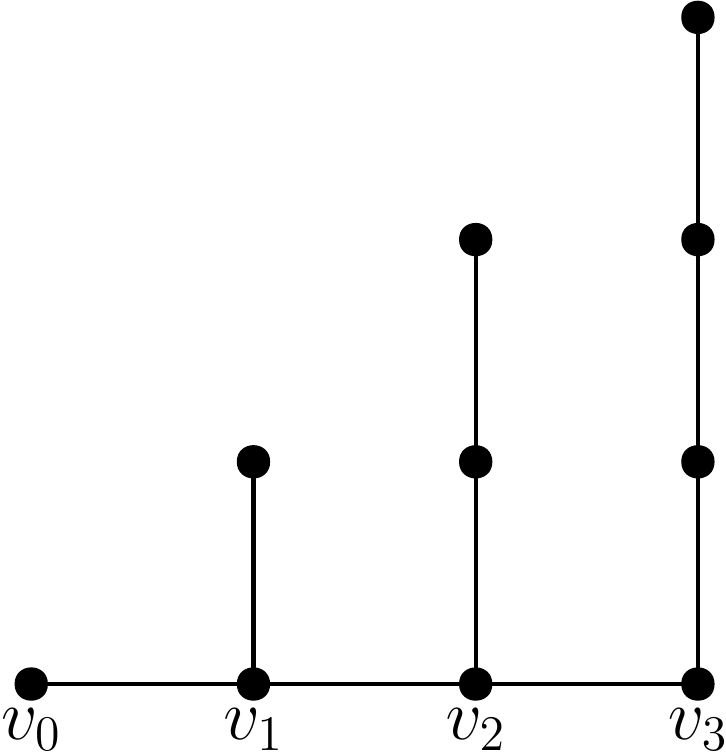}
\centering
\caption{The graph $F_3$.}
 \label{figure: F3}
\end{figure}

\begin{theorem}
\label{thm: construct_F}
For every $k\in \mathds{Z}^+ \cup \{0\}$, tree $F_k$ of order $\Theta(k^2)$ has ${\bdim}(F_k) =  O(k)$ and $\adim(F_k) = \Theta(k^2)$. 
\end{theorem}
\begin{proof} 
The function $f_k:V(F_k)\rightarrow \mathds{Z}^+ \cup \{0\}$ with $f_k(v_0) = f_k(v_k) =2k$ and $f_k(v)=0$ for all other vertices $v
\in V(F_k)$ is a resolving broadcast of $F_k$ with $c_{f_k}(F_k) = 4k$, so $\bdim(F_k)\leq 4k = O(k)$.

The size of any adjacency resolving set of $F_k$ must be linear in the number of vertices in order for all of the vertices on the paths attached to $L_k$ to be distinguished. Since tree $F_k$ has order $\Theta(k^2)$, we have $\adim(F_k) = \Theta(k^2)$.
\end{proof}

Combining \cref{corollary: acyclic} and \cref{thm: construct_F}, we have proven \cref{thm: asymptotic_lower_bound1}.

In \cite{geneson2020broadcast}, Geneson and Yi showed that, for two connected graphs $G$ and $H$ such that $H\subset G$, the ratios $\frac{\dim(H)}{\dim(G)}$, $\frac{\adim(H)}{\adim(G)}$, and $\frac{\bdim(H)}{\bdim(G)}$ can be arbitrarily large. In the next result, we show that this can only be true when the graph $G$ is not acyclic.

\begin{proposition}
For two trees $T_1$ and $T_2$ such that $T_1\subseteq T_2$, we have that $\dim(T_1)\leq \dim(T_2)$, $\adim(T_1)\leq \adim(T_2)$, and $\bdim(T_1)\leq \bdim(T_2)$. 
\end{proposition}
\begin{proof}
Let $T$ be a tree with efficient resolving broadcast $f:V(T)\rightarrow \mathds{Z}^+ \cup \{0\}$. Let $v\in V(T)$ be a leaf of $T$, and let $uv\in E(T)$.
If $v\not\in\supp_{T}(f)$, then $g:V(T-v)\rightarrow \mathds{Z}^+ \cup \{0\}$ with $g(w) = f(w)$ for every $w\in V(T-v)$ is a resolving broadcast of graph $T-v$. If $v\in\supp_{T}(f)$, then $g:V(T-v)\rightarrow \mathds{Z}^+ \cup \{0\}$ with $g(u) = \max\{f(v) -1, f(u)\}$ and $g(w) = f(w)$ for every $w\in V(T-v) - \{u\}$ is a resolving broadcast of $T-v$. Thus, $\bdim(T-v)\leq \bdim(T)$ for any leaf $v$ of $T$. 
Tree $T_2$ can be pruned into tree $T_1$ by repeatedly deleting leaves that are not in $T_1$. Thus, $\bdim(T_1)\leq \bdim(T_2)$.
The results $\dim(T_1)\leq \dim(T_2)$ and $\adim(T_1)\leq \adim(T_2)$ follow with similar reasoning.
\end{proof}

\section{Comparing $\adim(G)$ and $\bdim(G)$}
\label{Section: Comparing}

Geneson and Yi \cite{geneson2020broadcast} showed that, for the the $d$-dimensional grid graph $G_k = \Pi_{i=1}^{d} P_k$, we have $\bdim(G_k) = \Theta(k)$ and $\adim(G_k) = \Theta(k^d)$ for every $k\in \mathds{Z}^+$ and any $d\geq 1$, where the constants in the bounds depend on $d$.
In this section, we prove the following theorem.

\begin{theorem}
\label{thm: compare}
There exists a family of graphs $\set{G_k}_{k\in \mathds{Z}^+}$ with $\bdim(G_k) = \Theta(k)$ and $\adim(G_k) = 2^{\Omega(k)}$ for every $k\in \mathds{Z}^+$.
\end{theorem}

First, we recall the following graph notation.
We denote by $G[S]$ the subgraph of $G$ induced by $S\subseteq V(G)$. 
The \textit{Cartesian product} of graphs $G$ and $H$, denoted by $G \square H$, is the graph with vertex set $V(G) \times V(H):=\{(u_1,u_2)\mid u_1\in V(G), u_2\in V(H)\}$, where $(u_1,u_2)$ is adjacent to $(v_1,v_2)$ whenever $u_1 = v_1$ and $u_2v_2 \in E(H)$, or $u_2 = v_2$ and $u_1v_1\in E(G)$.

We prove \cref{thm: compare} by finding a family of graphs with the desired properties. This family of graphs is defined as follows:

\begin{definition}
Graph $\widehat{X_0}$ is the path $a,b,c$, and graph $\widehat{X}$ is the graph with vertex set $\{a,b,c\}$ and edge set $\{ab\}$.
For $i\in \mathds{Z}^+$, we let $$\widehat{X_i} = \widehat{X_0} \square \underbrace{\widehat{X} \square \widehat{X} \dots \square \widehat{X}}_{i\text{ times}}.$$ For $i\in \mathds{Z}^+ \cup \{0\}$, graph $X_i$ is $\widehat{X_i}$ with one modification: for every $1\leq j\leq i+1$, graph $X_i$ has an additional vertex $s_j$ that is adjacent to every vertex with $a$ as the $j$th coordinate. (See \cref{figure: G1} for the graph $X_1$.)
\end{definition}

\begin{figure}[t]
\includegraphics[scale=.5]{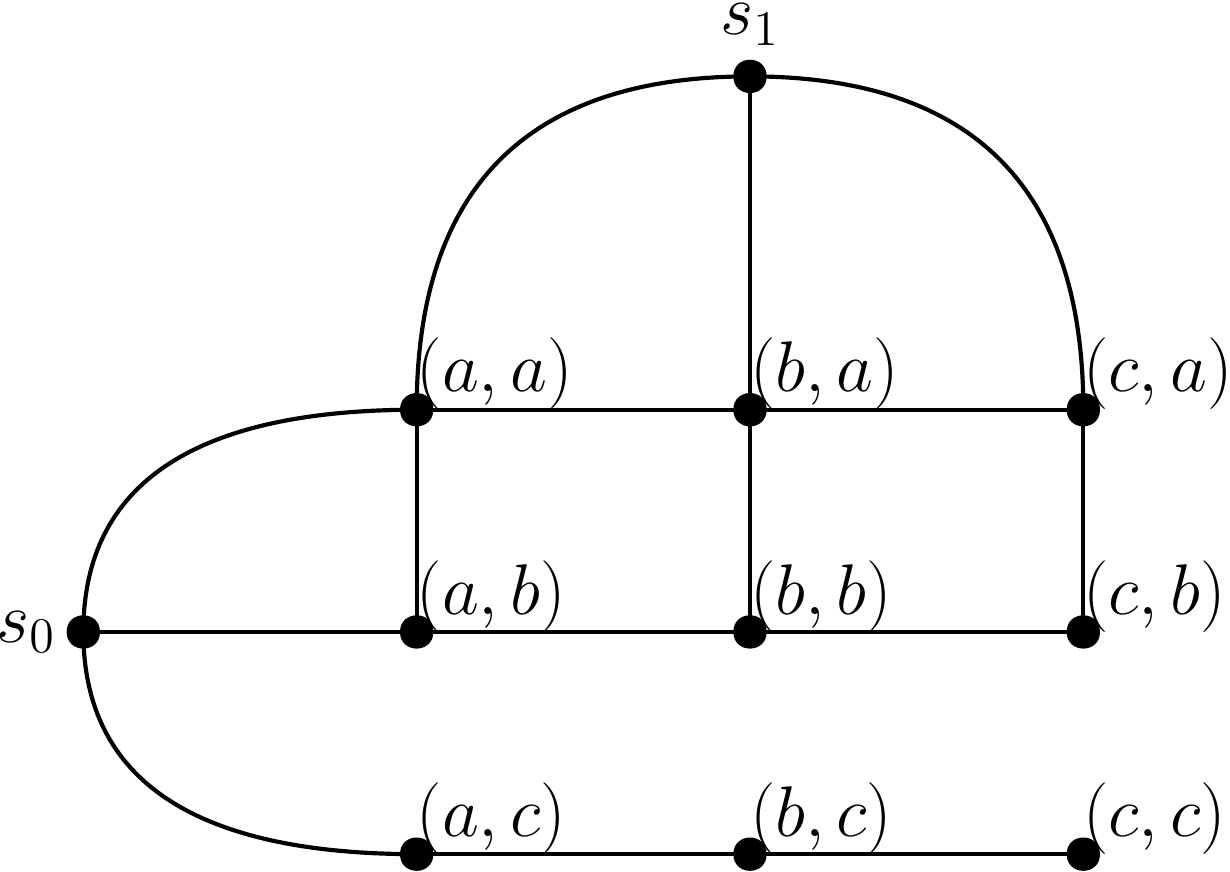}
\centering
\caption{The graph $X_1$.}
 \label{figure: G1}
\end{figure}

\begin{lemma} \label{lemma: compare1}
We have $\bdim(X_k) = \Theta(k)$ for all $k\in \mathds{Z}^+$. 
\end{lemma}
\begin{proof}
Let $k\in \mathds{Z}^+$ be given.
For $i\in \mathds{Z}^+ \cup \{0\}$, we define $S_i = \set{s_j\mid 0\leq j\leq i}$.

For $i\in \mathds{Z}^+ \cup \{0\}$, we define the function $f_i:V(X_i)\rightarrow \mathds{Z}^+ \cup \{0\}$ as follows: $f_i(s_0)=3$, $f_i(s_j)=2$ for every $1\leq j \leq i$, and $f_i(v)=0$ for all other vertices $v$.
We claim that $f_i$ is a resolving broadcast of $X_i$ for all $i\in \mathds{Z}^+ \cup \{0\}$. We proceed to prove this claim by induction. 

In the base case $i=0$, we have that $X_0$ is the path $s_0,a,b,c$. It is easy to see that that function $f_0$ with $f_0(s_0) =3$ is a resolving broadcast of $X_0$. Assuming that $f_{k-1}$ is a resolving broadcast of graph $X_{k-1}$, we will show that $f_k$ is a resolving broadcast of graph $X_k$. 

Let $u_1, u_2\in V(X_{k-1})$ and $v_1, v_2 \in \{a,b,c\}$ such that $(u_1, v_1)$ and $(u_2, v_2)$ are two distinct vertices in $V(X_k)$. If $u_1\neq u_2$, then $(u_1, v_1)$ and $(u_2, v_2)$ are resolved by the vertex in $S_{k-1}$ that resolved $u_1$ and $u_2$ in $X_{k-1}$.
Alternatively, if $u_1 = u_2$, then $v_1\neq v_2$, and $(u_1, v_1)$ and $(u_2, v_2)$ are  resolved by $s_k$. 
Thus, function $f_k$ is a resolving broadcast of $X_k$.

Now we can upper bound the broadcast dimension of $X_k$: $$\bdim(X_k)\leq c_{f_k}(X_k) = 3 + 2k \implies \bdim(X_k)= O(k).$$
By \cref{thm: lowerbound_adim_bdim}, we have $\bdim(X_k) = \Omega(k)$. Thus, we have $\bdim(X_k) = \Theta(k)$.
\end{proof}

\begin{lemma} \label{lemma: compare2}
We have $\adim(X_k) = 2^{\Omega(k)}$ for all $k\in \mathds{Z}^+$. 
\end{lemma}
\begin{proof}
Let $k\in \mathds{Z}^+$ be given.
For $i\in \mathds{Z}^+ \cup \{0\}$, we define $S_i = \set{s_j\mid 0\leq j\leq i}$.

We claim that the following statement is true for all $i\in \mathds{Z}^+ \cup \{0\}$: for any adjacency resolving set $A_i$ of $X_i$, we have that $|(V(X_i)-S_i)\cap A_i| \geq 2^i$. We proceed to prove this claim by induction.

In the base case $i=0$, for any adjacency resolving set $A_0$ of $X_0 = P_4$, we have by \cref{thm: paths_adj} 
$$|(V(X_0)-\{s_0\})\cap A_0| \geq \floor{\frac{2(4)+2}{5}}-1 = 1.$$ 
Now we assume that $|(V(X_{k-1})-S_{k-1})\cap A_{k-1}| \geq 2^{k-1}$ for any adjacency resolving set $A_{k-1}$ of $X_{k-1}$.

Let $H$ be $X_{k-1}[V(X_{k-1})-S_{k-1}]$, the subgraph induced in $X_{k-1}$ by $V(X_{k-1})-S_{k-1}$. 
The induced subgraph $X_k[V(X_{k})-S_{k}]$ contains three copies of $H$ as subgraphs. Let $H_1$, $H_2$, and $H_3$ be the copies of $H$ in $X_k[V(X_{k})-S_{k}]$ that are induced by the sets of vertices $\{(v,a)\mid v\in V(H)\}$,  $\{(v,b)\mid v\in V(H)\}$, and $\{(v,c)\mid v\in V(H)\}$, respectively.

Let $v\in V(H)$. Vertex $(v,c) \in V(H_3)$ is only adjacent to vertices in $V(X_k) - V(H_3)$ that are in $S_{k-1}$. In particular, the vertices $(v,c) \in V(H_3)$ and $u\in S_{k-1}$ are adjacent in $X_k$ if and only if $v$ and $u$ are adjacent in $X_{k-1}$. Thus, we have $\left|V(H_3)\cap A_{k}\right| \geq 2^{k-1}$ for any adjacency resolving set $A_k$ of $X_k$ by the inductive hypothesis.

If $\left|V(H_1)\cap A_{k}\right| = 0$, then we must have $\left|V(H_2)\cap A_{k}\right| \geq 2^{k-1}$, in order to distinguish all of the vertices in $H_2$. If instead $\left|V(H_1)\cap A_{k}\right| = x$ for some positive integer $x$, then we must have $\left|V(H_2)\cap A_{k}\right| \geq 2^{k-1}-x$, since every vertex in $V(H_1)\cap A_k$ reaches at most one vertex in $H_2$.
Thus, any adjacency resolving set $A_k$ of $X_k$ must have at least $2^{k-1}$ vertices in $V(H_1)\cup V(H_2)$. 
We have $$\left|(V(X_{k})-S_{k})\cap A_{k}\right| = |V(H_3) \cap A_k| + |(V(H_1)\cup V(H_2))\cap A_k| \geq 2^k$$  for any adjacency resolving set $A_k$ of $X_k$, which completes the induction.

Thus, we have $|A_k| \geq 2^k$ for any adjacency resolving set $A_k$ of $X_k$, so $\adim(X_k) = 2^{\Omega(k)}$. 
\end{proof}


Combining \cref{lemma: compare1} and \cref{lemma: compare2}, we have proven \cref{thm: compare}. 
We note that our construction of graph $X_k$ has broadcast dimension that is asymptotically optimal in both its order and its adjacency dimension:

\begin{remark}
There does not exist a family of graphs $\set{G_k}_{k\in \mathds{Z}^+}$ with $\bdim(G_k) = \Theta(k)$ and $\adim(G_k) = 2^{\omega(k)}$ for every $k\in \mathds{Z}^+$ because $\bdim(G) = \Omega(\log n)$ for all graphs $G$ of order $n$ by \cref{thm: lowerbound_adim_bdim}.
\end{remark}

Our result in \cref{thm: compare} directly implies \cref{thm: asymptotica_lower_bound2} and resolves \cref{Question: 1} affirmatively.
Furthermore, we can also answer \cref{Question: 1} for acyclic graphs:

\begin{remark}
By \cref{corollary: adim_of_acyclic2}, there does not exist a family of acyclic graphs $\set{G_k}_{k\in \mathds{Z}^+}$ with $\bdim(G_k) = \Theta(k)$ and $\adim(G_k) = 2^{\Omega(k)}$ for every $k\in \mathds{Z}^+$.
\end{remark}

\section{Edge Deletion}
\label{Section: Edge_Deletion}
Throughout this section, we let $v$ and $e$, respectively, denote a vertex and an edge of a connected graph $G$ such that $G-v$ and $G-e$ are also connected graphs.
Geneson and Yi \cite{geneson2020broadcast} constructed families of graphs that demonstrated that both $\frac{\bdim(G)}{\bdim(G-v)}$ and $\bdim(G-v) - \bdim(G)$ can be arbitrarily large. 
In this section, we prove analogues of their results for the effect of edge deletion on the broadcast dimension of a graph.
We prove \cref{thm: edge_deletion1} and \cref{thm: edge_deletion2}, which state that
both $\bdim(G) - \bdim(G-e)$ and
$\bdim(G-e) - \bdim(G)-d_{G-e}(u,v)$ can be arbitrarily large for $e=uv\in E(G)$. We do so by finding families of graphs that demonstrate these results. We also show that the ratio $\frac{\bdim(G-e)}{\bdim(G)}$ is bounded from above by 3, proving \cref{thm: last}.

In the following theorem, we resolve \cref{Question: 2} affirmatively by constructing a family of graphs that uses ideas from a graph constructed by Eroh et al. in \cite{eroh2015effect}, which they used to show that $\dim(G) - \dim(G-e)$ can be arbitrarily large.

\begin{reptheorem}{thm: edge_deletion1}
The value $\bdim(G) - \bdim(G-e)$ can be arbitrarily large. 
\end{reptheorem}

\begin{figure}[t]
\centering
\begin{subfigure}{.5\textwidth}
  \centering
  \includegraphics[scale=.5]{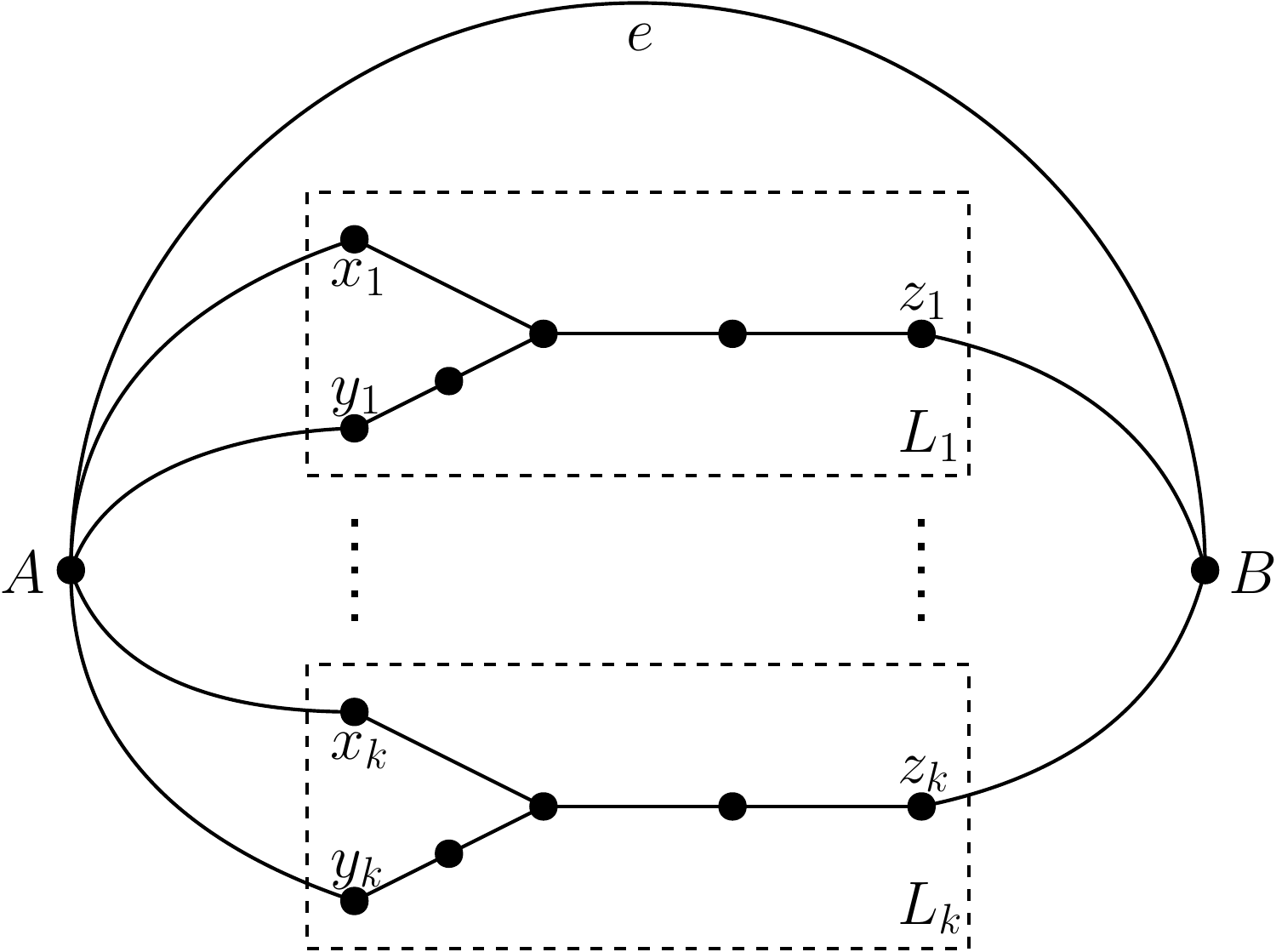}
\end{subfigure}%
\begin{subfigure}{.5\textwidth}
  \centering
  \includegraphics[scale=.5]{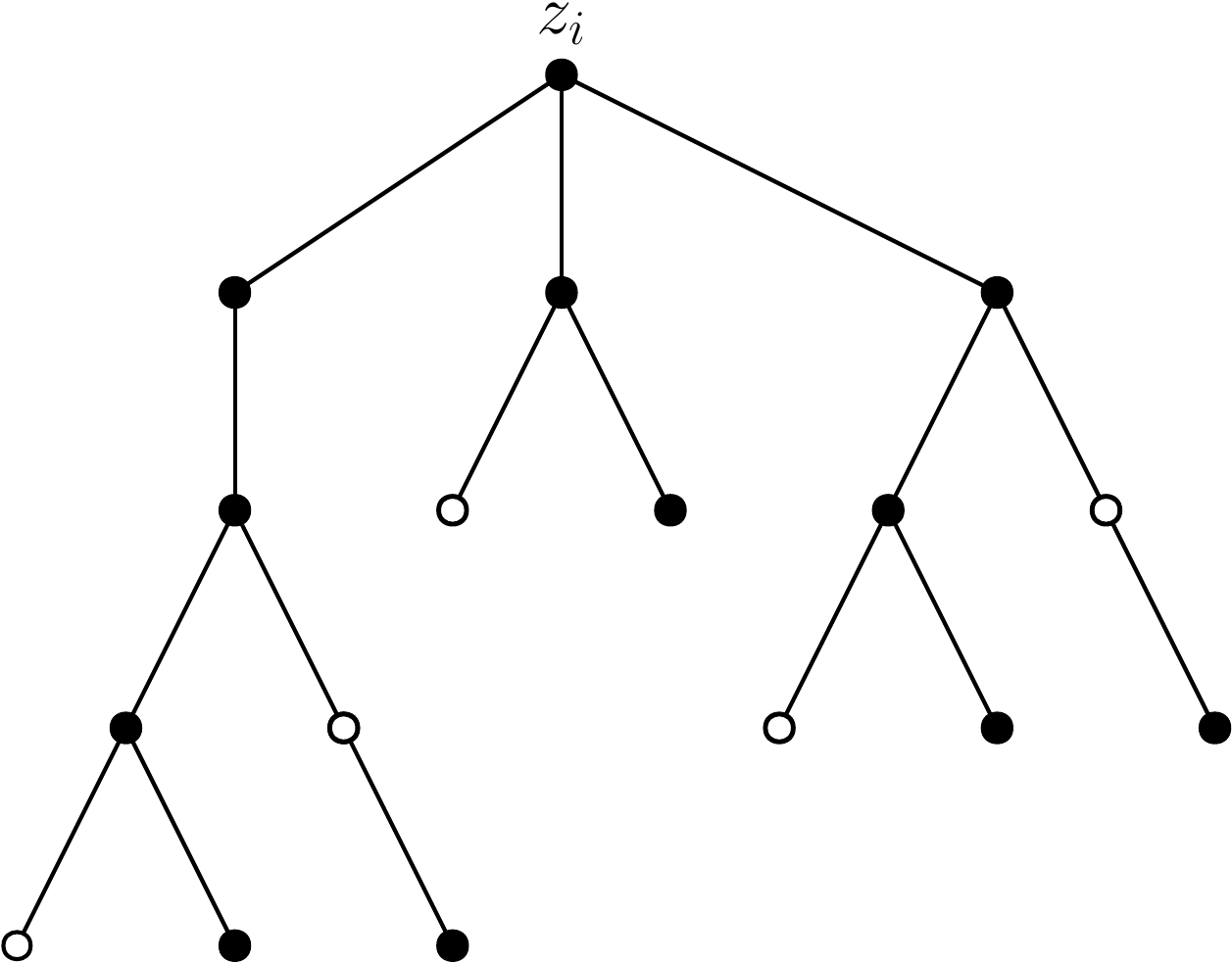}
\end{subfigure}
\begin{minipage}{.5\textwidth}
  \centering
  \vspace*{.5cm}
  $G_k$
\end{minipage}%
\begin{minipage}{.5\textwidth}
  \centering
    \vspace*{.5cm}
  $T_i$
\end{minipage}
\caption{A graph $G_k$ such that $\bdim(G_k) - \bdim(G_k-e) = \Omega(k)$. For every $1\leq i\leq k$, vertex $z_i$  is the root of a copy of tree $T_i$, shown on the right, so $|L_i| = 22$.}
\label{figure: DeleteEdge1}
\end{figure}

\begin{proof}
Let $k\in \mathds{Z}^+$, and let $G_k$ be the graph in \cref{figure: DeleteEdge1} with $e = AB$.  For each $1\leq i \leq k$, let \textit{layer} $L_i$ be the set of vertices indicated in \cref{figure: DeleteEdge1}. 

We define function $g:V(G_k-e)\rightarrow \mathds{Z}^+ \cup \{0\}$ as follows: 
$$
g(v)=
\begin{cases}
3 &\text{ if } v=A,\\
4 &\text{ if } v=z_i\text{ for }1\leq i \leq k,\\
1  &\text{ if } v \text{ is a vertex on tree }T_i\text{ shown with an open circle in \cref{figure: DeleteEdge1}},\\
0 &\text{ otherwise. }
\end{cases}
$$
Because $g$ is a resolving broadcast of graph $G_k-e$, we can upper bound the broadcast dimension of graph $G_k-e$: we have $\bdim(G_k-e) \leq c_g(G_k-e)=3 + 9k.$

Let $f:V(G_k)\rightarrow \mathds{Z}^+ \cup \{0\}$ be a resolving broadcast of the graph $G_k$.
For every pair of distinct vertices $u_1,u_2 \in V(G_k)$ with $d(u_1, A) = d(u_2,A)$ and $d(u_1, B) = d(u_2,B)$, at least one of $u_1$ or $u_2$ must be reached by a vertex in $\supp_{G_k}(f)$ that is on the same layer since otherwise we would have $b_f(u_1) = b_f(u_2)$.
Thus, at most $$\max_{u,v\in V(G_k)}(d(u,A)+1)(d(v,B)+1) + 1 = O(1)$$ layers of graph $G_k$ can have a vertex that is not reached by any vertex on the same layer. The following properties must hold for the remaining $k - O(1)$ layers $L_i$.

\begin{enumerate}
\item Every vertex $v\in L_i$ is reached by a vertex in $L_i\cap \supp_{G_k}(f)$.
\item We have $\supp_{G_k}(f) \cap (L_i - V(T_i)) \neq \emptyset$, since otherwise $b_f(x_i) = b_f(y_i)$. 
\item Any distinct $u,v\in V(T_i)$ with $d(u,z_i) =d(v,z_i)$ must be resolved by a vertex in $L_i$ since $d(u,A) = d(v,A)$ and $d(u,B) = d(v,B)$.
\end{enumerate}

\begin{figure}[h]
\centering
\includegraphics[scale=.5]{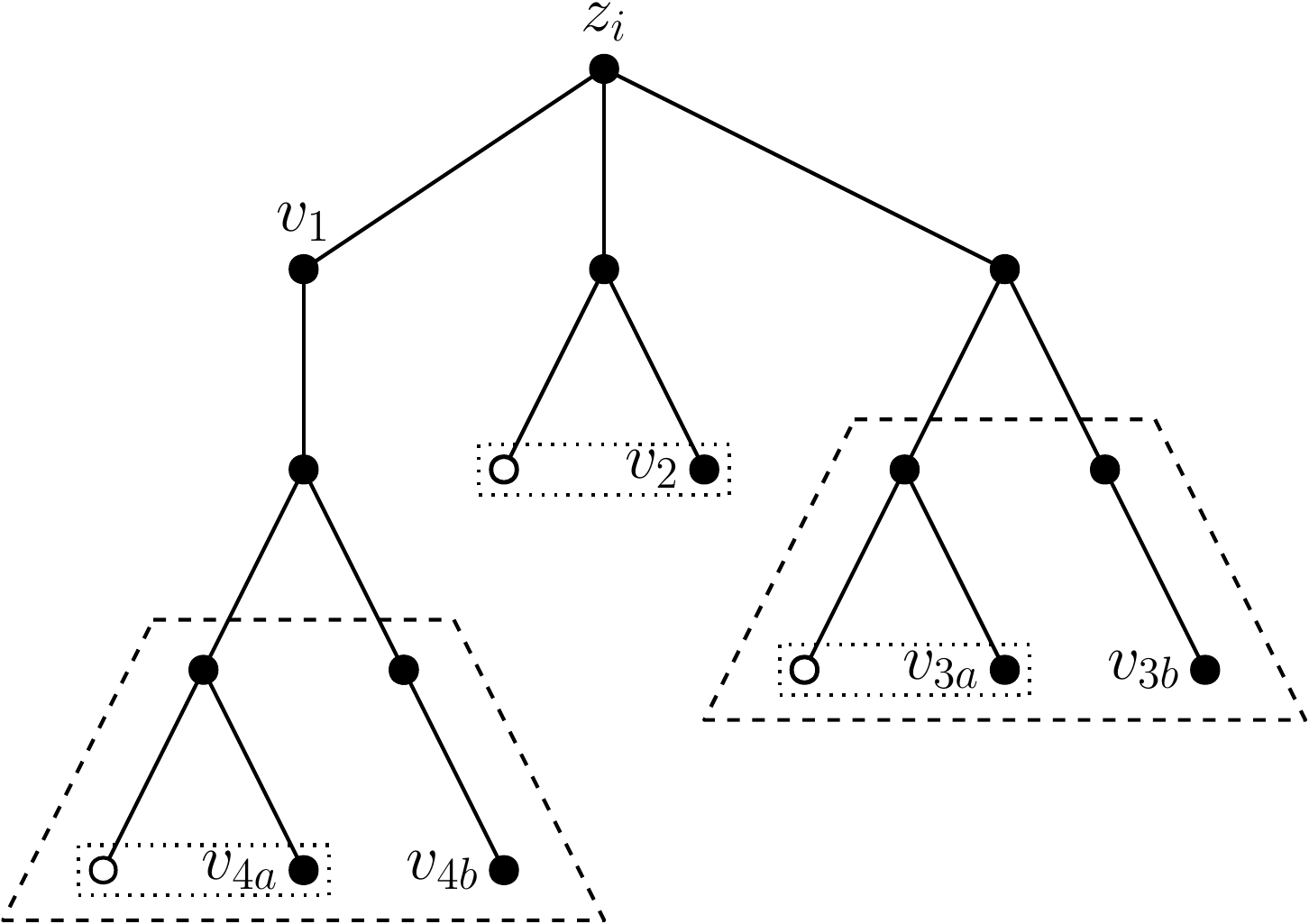}
\caption{Tree $T_i$ labeled for casework reference.}
 \label{figure: ForceFourLabeled}
\end{figure}
Refer to \cref{figure: ForceFourLabeled} for the remainder of the proof.
There are three pairs of twin vertices on tree $T_i$ (see dotted rectangular boxes). By \cref{obs: twin}, at least one of the vertices in each of these pairs must be in $\supp_{G_k}(f)$. Without loss of generality, let the three vertices that are denoted with an open circle be in $\supp_{G_k}(f)$.
The total value assigned to each of the two groups of five vertices identified by dashed trapezoidal boxes must be at least 2 in order for the three vertices that are the same distance away from $z_i$ in each of those groups to be distinguished.

Any assignment of a total value of 5 to $T_i$ subject to the above constraints leaves at least four unreached vertices: $v_1$, $v_2$, $v_{3a}$ or $v_{3b}$, and $v_{4a}$ or $v_{4b}$. These four vertices must be reached by assigning an additional total value of at least $4$ to the vertices on tree $T_i$ (in addition to the positive value assigned to some vertex in $L_i - V(T_i)$), or by assigning an additional total value of $c< 4$ to the vertices on tree $T_i$ and at least $5-c$ to a vertex $v\in L_i- V(T_i)$. In either case, a total value of at least 10 must be assigned to the vertices on such a layer $L_i$.

Because a total value of at least 10 is assigned to at least $k - O(1)$ layers of $G_k$ by any resolving broadcast $f$ of $G_k$, we have $\bdim(G_k) - \bdim(G_k-e) \geq 10k - 9k  - O(1)= \Omega(k)$.
\end{proof}

We will prove \cref{thm: edge_deletion2} by constructing a family of graphs that shows that $\bdim(G-e) - \bdim(G)$ can be arbitrarily larger than $d_{G-e}(u,v)$, thus showing that the bound proposed in \cref{Question: 3} can fail. Since we use both spider graphs and the graph $F_k$ (see \cref{def: F_k}) in the graph construction, we begin with two lemmas: a lemma about graphs containing spiders as subgraphs and a lemma about the graph $F_k$. 

As we will be working with a specific family of spider graphs in the proof \cref{thm: edge_deletion2}, we introduce our notation for spider graphs:

\begin{definition}
A tree is called a \textit{spider} if every vertex, except for one vertex known as the \textit{center vertex}, has degree at most two. A \textit{leg} of a spider graph is a path connected to the center vertex. We denote by $SP\paren{\ell_1^{(x_1)}, \dots, \ell_m^{(x_m)}}$ a spider of order $n$ with $x_i$ legs of length $\ell_i$ for every $1\leq i\leq m$, where $\ell_1\leq \ell_2\leq \dots \leq \ell_m$ and $1+ \sum_{i=1}^{m}\ell_ix_i = n$.
\end{definition}

\begin{lemma}
\label{lemma: spider}
For any integer $k>1$, let $G$ be a graph that contains spider $SP\paren{3k^{\paren{6k}}}$ with center $c$ as a subgraph such that $c$ is the only vertex on the spider that is adjacent to any vertex of the graph that is not on the spider. If a resolving broadcast $f$ of $G$ is efficient, then there exists a vertex $z\in \supp_G(f)$ with $f(z) -d(c,z) \geq 3k-2$.
\end{lemma}
\begin{proof}

For the sake of contradiction, consider an efficient resolving broadcast $f$ where there is not a $z\in \supp_G(f)$ with $f(z) -d(c,z) \geq 3k-2$.  On each leg of the spider $SP\paren{3k^{\paren{6k}}}$, the three vertices farthest from $c$ are only reached by vertices on the same leg.

Let $u,v\in V(G)$ be two distinct vertices on the legs of the spider with $d(u,c) = d(v,c)$. If neither $u$ nor $v$ are reached by a vertex that is on the same leg as $u$ or $v$, then $b_f(u) = b_f(v)$. Thus, every vertex on the legs of the spider, except at most $3k$ of them (one vertex of each distinct distance from $c$) must be reached by another vertex on the same leg. On at least $6k-3k=3k$ legs of the spider, all vertices need to be reached by a vertex on the same leg; let $L$ be the set of these legs.

A vertex $v\in \supp_G(f)$ on a leg of the spider can reach at most $2f(v) + 1$ vertices on the same leg, and we have that $2f(v)+1 \leq 3f(v)$ with equality if and only if $f(v)=1$.
Because all of the vertices on a leg $\ell\in L$ need to be reached by a vertex on $\ell$, the total value $v_{\ell}$ assigned to the $3k$ vertices on $\ell$ must be at least $k$.
If $v_{\ell} = k$, then we must have the following assignment: the vertices on $\ell$ that are distance $3i-1$ from $c$ for $1\leq i \leq k$ are assigned a value of 1, and the rest of the vertices on $\ell$ are assigned 0. However, with this assignment, there are two vertices on $\ell$ that have the same broadcast representation: the vertex that is distance $3k-2$ from $c$ and the vertex that is distance $3k$ from $c$ are only reached by the vertex between them. Thus, $v_{\ell} \geq k+1$ for every $\ell\in L$.

Consider function $f':V(G)\rightarrow \mathds{Z}^+ \cup \{0\}$ defined as follows. Let $f'(v) = f(v)$ for all vertices that are not on a leg in $L$, and let $f'(c)=3k-2$. The vertices on a leg in $L$ that are distance $3i-1$ from $c$ for $1\leq i \leq k$ are assigned a value of 1, and the rest of the vertices on a leg in $L$ are assigned 0. We note that $f'$ is a resolving broadcast of $G$. Moreover, we have $c_{f'}(G) < c_{f}(G)$ because $f'(c) - f(c) \leq 3k-2$, and for each of the $3k$ legs $\ell\in L$, we have $\sum_{v\in \ell} f(v) -  \sum_{v\in \ell} f'(v) \geq 1$. This contradicts the efficiency of $f$.
\end{proof}

\begin{lemma}
\label{lemma: F_k}
For any resolving broadcast $f$ of $F_k$ (see \cref{def: F_k}) with $f(v_k) \geq k$, we have $$\sum_{v\in V(F_k)} f(v) \geq f(v_k) + 2k - O(1).$$
\end{lemma}
\begin{proof}
Let $f$ be a resolving broadcast of $F_k$ that minimizes $\sum_{v\in V(F_k)} f(v)$, under the constraint that $f(v_k) \geq k$.
For $u\in V(F_k)$, we define $p(u) := \argmin_{i\in [0,k]}d\paren{u,v_i}$.



\noindent\textbf{Case 1.}  $f(w)- d\paren{w,v_{p(w)}} \leq \ceil{\frac{k}{2}}$ for all $w \in V(F_k) - \{v_k\}$.\\
Define $f': V(F_k)\rightarrow \mathds{Z}^+ \cup \{0\}$ such that $f'(v_k) = f(v_k)$ and $f'(v) = 0$ for all other $v\in V(F_k)$. The number of unique broadcast representations of the vertices of $F_k$ with respect to function $f'$, denoted $B_{F_k}(f')$, is $k+1$.
Updating $f'(w)\leftarrow f(w)$ for any $w \in \supp_{F_k}(f) - \{v_k\}$
introduces at most $f(w)$ new unique broadcast representations to the vertices $u$ with $p(u) < p(w)$. Thus, every update $f'(w)\leftarrow f(w)$ for some $w \in \supp_{F_k}(f) - \{v_k\}$ increases $B_{F_k}(f')$ by at most $$O\paren{f(w)}+ \sum _{i=1}^{f(w)} i \leq  f(w)\paren{\frac{k}{4}+ O(1)}.$$ 
Since we must have $\frac{k^2}{2} + O(k)$ unique broadcast representations, the lemma holds in this case.


\noindent\textbf{Case 2.} There is a vertex $w
\neq v_k$ with $f(w) - d\paren{w,v_{p(w)}}>  \ceil{\frac{k}{2}}$.\\
Let $t$ be the vertex on $F_k$ farthest from vertex $v_0$.
We must have $d(w,t) -f(w)=O(1)$, since otherwise there would be $\omega(1)$ vertices $u$ with $p(u) \geq p(w)$ not reached by $w$. These vertices would be most efficiently distinguished by increasing $f(w)$, contradicting the efficiency of $f$. 
The vertices $u$ with $p(u) < p(w)$ must be distinguished with an additional total cost of at least $p(w) - O(1)$. Thus, 
\begin{align*}
\sum_{v\in V(F_k)} f(v) \geq f(v_k) + d(w,t) + p(w) - O(1) \geq f(v_k)+2k - O(1),
\end{align*}
as desired.
\end{proof}

With \cref{lemma: spider} and \cref{lemma: F_k}, we can prove \cref{thm: edge_deletion2}.

\begin{reptheorem}{thm: edge_deletion2}
The value $\bdim(G-e) - \bdim(G)$ can be arbitrarily larger than $d_{G-e}(u,v)$, where $e=uv\in E(G)$.
\end{reptheorem}

\begin{proof}
For integer $k\geq 2$, let $H_k$ be the graph in \cref{figure: DeleteEdge2_}, and let $e=v_{i}v_{i+1}$, where $i = \floor{\frac{3k-2}{2}}$. Let $S_1$ and $S_2$ be the spider $SP\paren{3k^{(6k)}}$ centered at $v_0$ and $v_{3k -2}$, respectively.
For $u\in V(H_k)$, we define $p(u) := \argmin_{i\in [0,k]}d_{H_k}(u,v_i)$.
We will show that for sufficiently large $k$, we have $$\bdim\paren{H_k-e} - \bdim(H_k) = d_{H_k-e}(v_{i},v_{i+1})+\Omega(k) = \frac{k}{2} +\Omega(k).$$

\begin{figure}[b]
\hspace*{.6cm}\includegraphics[scale=1]{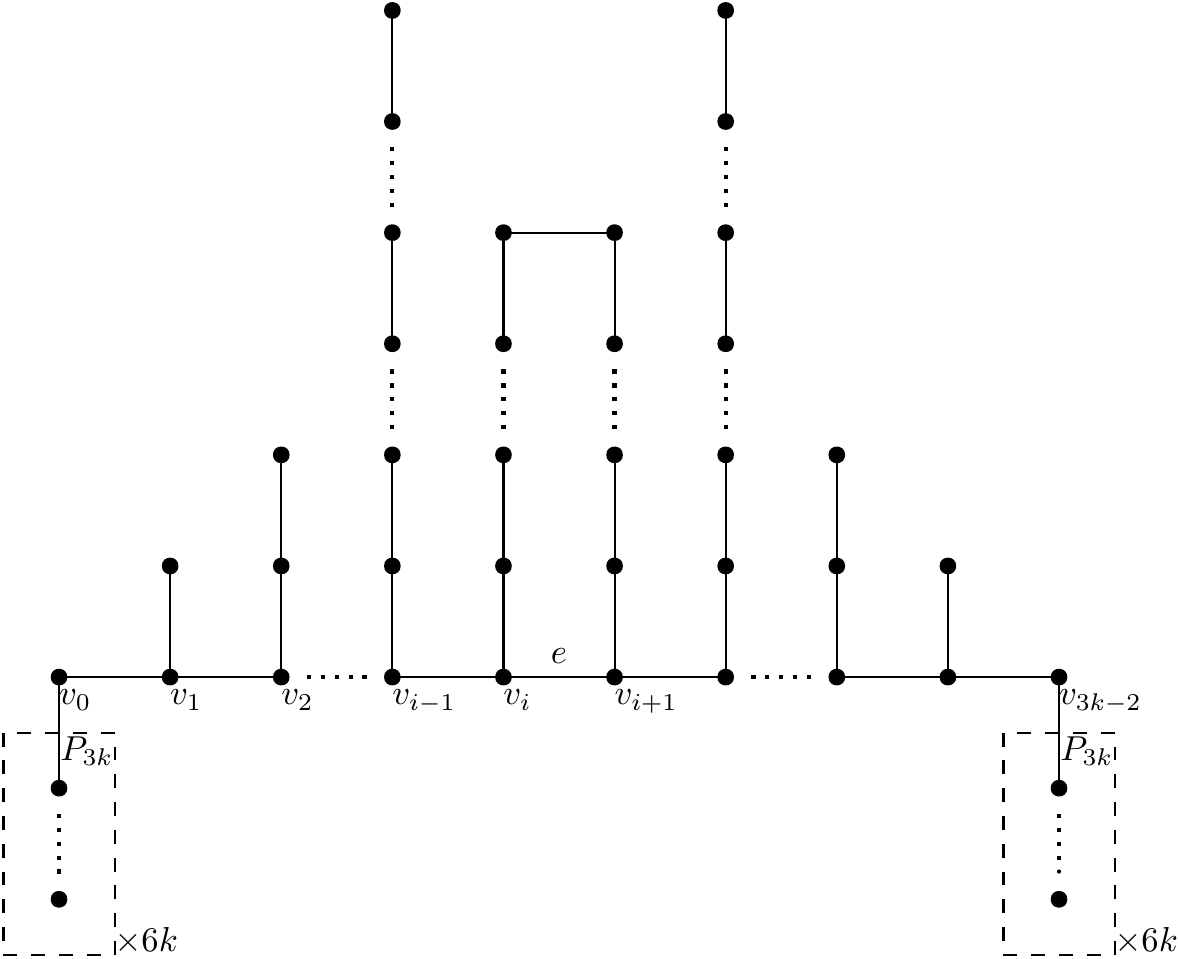}
\caption{A graph $H_k$ such that $\bdim(H_k-e) - \bdim(H_k)$ can be arbitrarily larger than $d_{H_k-e}(v_i,v_{i+1})$, where $e=v_iv_{i+1}$ and $i = \floor{\frac{3k-2}{2}}$.
The vertices $v_0, \dots, v_{3k-2}$ are on a path. Additionally, each $v_j$ with $1\leq j \leq i-1$ is connected to a path $P_j$; each $v_j$ with $i+2\leq j \leq 3k-3$ is connected to a path $P_{3k-2-j}$, and vertices $v_i$ and $v_{i+1}$ are on a cycle of length $\floor{\frac{k}{2}}$. Finally, $v_0$ and $v_{3k-2}$ are both centers of a copy of spider $SP\paren{3k^{(6k)}}$.}
 \label{figure: DeleteEdge2_}
\end{figure}

Let $B = \bdim\paren{SP\paren{3k^{\paren{6k}}}}$. Let $g:V(H_k)\rightarrow \mathds{Z}^+ \cup \{0\}$ be the function that applies an efficient resolving broadcast of $SP\paren{3k^{(6k)}}$ to $S_1$ and $S_2$ on graph $H_k$. 
By \cref{lemma: spider}, there are vertices $z_1'$ on $S_1$ and $z_2'$ on $S_2$ with $g(z_1') - d_{H_k}(v_0, z_1')\geq 3k-2$ and $g(z_2') - d_{H_k}(v_{3k-2}, z_2')\geq 3k-2$. 
Function $g$ is a resolving broadcast of $H_k$ since every pair of distinct vertices in $V(H_k)$ that are on the same spider is clearly resolved, and every other pair of vertices is resolved by either $z_1'$ or $z_2'$. Thus, $\bdim(H_k)\leq 2B$. 

Let $f$ be an efficient resolving broadcast of the graph $H_k-e$. By \cref{lemma: spider}, we must have vertices $z_1,z_2\in \supp_G(f)$ with $f(z_1) - d_{H_k-e}(v_0, z_1)\geq 3k-2$ and $f(z_2) - d_{H_k-e}(v_{3k-2}, z_2)\geq 3k-2$. 


\noindent\textbf{Case 1.} There does not exist a vertex $z$ with $$f(z) - d_{H_k-e}(v_0, z)\geq 3k-2 \quad\text{and}\quad f(z) - d_{H_k-e}(v_{3k-2}, z)\geq 3k-2.$$
In this case, $z_1$ and $z_2$ are distinct vertices.
We define $c_1 := f(z_1) - 3k +2$ and $d_1 :=p(z_1)$, and we similarly define $c_2 := f(z_2) - 3k +2$ and $d_2:= 3k-2-p(z_2)$. Note that $c_1 \geq d_1$ and $c_2 \geq d_2$ by  \cref{lemma: spider}. 

If $d_1>0$, let $T_1$ be the $F_{d_1}$ subgraph induced by the vertices $u$ with $p(u)\leq d_1$ that are not on a leg of spider $S_1$. Otherwise, let $T_1$ be the graph that consists of the singular vertex $z_1$.
Let $y_1 = \argmax_{V(S_1) - \{v_0\}}(f(y) - d(y, v_0))$. 
By \cref{lemma: F_k}, we must have 
\begin{equation} \label{eqn: F_k}
f(y_1) - d(y_1,v_0)+ \sum_{v\in V(T_1)} f(v) \geq f(z_1) + 2d_1 + \max\set{f(y_1) - d(y_1,v_0) -2d_1, 0} - O(1)
\end{equation}
in order to distinguish the vertices of $T_1$. 

On spider $S_1$, assigning vertex $y_1$ the value $f(y_1)$ only distinguishes at most $d(y_1, v_0) + f(y_1)$ vertices on the same leg of $S_1$. 
In an efficient resolving broadcast of $S_1$, those $d(y_1, v_0) + f(y_1)$ vertices would have instead been distinguished with a total cost of at most $\ceil{\frac{d(y_1, v_0) + f(y_1)}{3}}$ by assigning a value of 1 to every third vertex (see proof of \cref{lemma: spider}). 
Thus, we can obtain the following bound on the total value assigned to the vertices in set $U_1 = V(S_1) - \{v_0, y_1, z_1\}$:
\begin{equation} \label{eqn: spider}
 \sum_{v\in U_1} f(v) \geq B - g(z_1') - \frac{d(y_1, v_0) + f(y_1) }{3} - O(1).
\end{equation}

Using \eqref{eqn: F_k} and \eqref{eqn: spider}, we lower bound the total value assigned to all of the vertices $u$ with $p(u) \leq d_1$:
\begin{align*}
& f(y_1) + \sum_{v\in V(T_1)} f(v) + \sum_{v\in U_1} f(v) \\
\geq \text{ }  &{f(z_1)} +2d_1+\max\set{ f(y_1)- d(y_1,v_0)- 2d_1, 0}+d(y_1,v_0) + {B -3k - \frac{d(y_1, v_0) + f(y_1) }{3}} - O(1) \\
\geq \text{ }  & c_1+2d_1+d(y_1,v_0) + {B  - \frac{d(y_1, v_0) + \paren{d(y_1,v_0)+ 2d_1 }}{3}} - O(1) \\
\geq \text{ }& {c_1} +\frac{ 4d_1 }{3}
+ B   - O(1). 
\end{align*}

In this case, the sum of the values assigned to vertices $u$ with $p(u) > d_1$ must be at least $B$ in order to distinguish the vertices of spider $S_2$. Thus, we have $$\bdim(H_k-e) - \bdim(H_k) \geq \paren{\sum_{v\in V(H_k - e)} f(v)} - 2B \geq c_1 + \frac{4}{3}d_1 - O(1).$$

If $d_1\geq \frac{k}{4}$, then we have 
$$\bdim(H_k-e) - \bdim(H_k) \geq c_1 + \frac{4d_1}{3} - O(1)\geq  \frac{7d_1}{3} - O(1) \geq \frac{7k}{12} - O(1)= \frac{k}{2} +\Omega(k),$$ as desired. By symmetry, if $d_2 \geq \frac{k}{4}$, then we are also done.

Now, we consider $d_1, d_2 < \frac{k}{4}$. The $\frac{k^2}{2} \pm O(k)$ vertices in region $A_2$ (see \cref{figure: DeleteEdge2_geo}) are all reached by $z_2$, and all but $O(k)$ of them must be reached by another vertex that is not in $B_2$ in order to be distinguished. Additionally, at least $\frac{k^2}{2} - O(k)$ of the vertices in $A_1$ must be reached by vertices not in $B_1$ in order to be distinguished. 
Vertex $z_1$ reaches at most $(c_1+d_1 + O(1))k $ of the vertices in $A_2$, and the total value assigned to the vertices in $B_1$ is at least $B + c_1 + \frac{4}{3}d_1 - O(1)$. 
Similarly, vertex $z_2$ reaches at most $(c_2+d_2 + O(1))k$ of the vertices in $A_1$, and the total value assigned to the vertices in $B_2$ is at least $B+ c_2 + \frac{4}{3}d_2 - O(1)$.
Any vertex $v$ that is not in $B_1$ or $B_2$ has $g(v) = 0$ and reaches at most $k \cdot f(v) + O(1)$ of the vertices in $A_1 \cup A_2$.
Thus, in this case we have $$\bdim(H_k-e) - \bdim(H_k)\geq \frac{1}{k}\paren{\abs{A_1 \cup A_2} - O(k)}= k - O(1) = \frac{k}{2} +\Omega(k).$$

 \begin{figure}[h]
\centering
\includegraphics[scale=1]{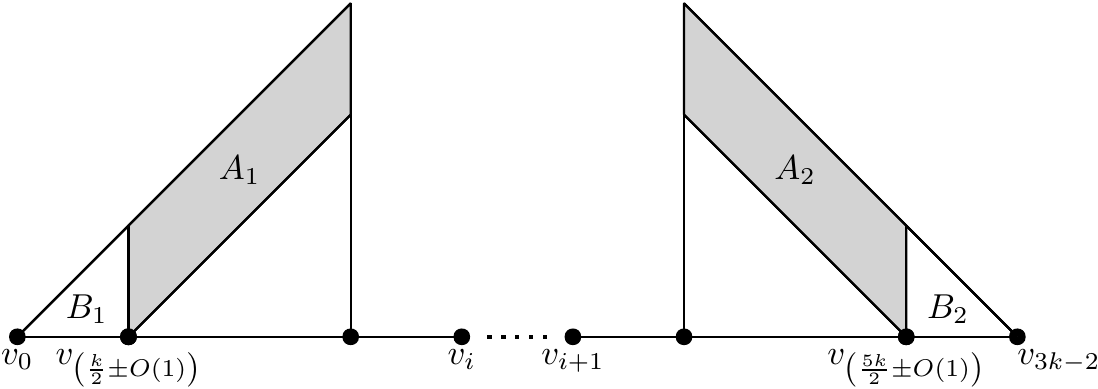}
\caption{A geometric interpretation of graph $H_k - e$. The spiders centered at $v_0$ and $v_{3k-2}$ (not pictured) are also in $B_1$ and $B_2$, respectively.}
\label{figure: DeleteEdge2_geo}
\end{figure}

\noindent\textbf{Case 2.} There exists a vertex $z$ with 
$$f(z) - d_{H_k-e}(v_0, z)\geq 3k-2 \quad \text{and}\quad f(z) - d_{H_k-e}(v_{3k-2}, z)\geq 3k-2.$$
The assumption in this case directly implies that $$f(z) \geq \frac{d_{H_k-e}(v_0,v_{3k-2})}{2} + 3k - O(1) = 4.75k - O(1).$$ Without loss of generality, we assume $p(z) \geq i$.
Let $T_1$ be the $F_{i-1}$ subgraph induced by the vertices $u$ with $p(u)\leq i-1$ that are not on a leg of spider $S_1$. By the same reasoning as in the first case, in order for the vertices on $T_1$ to be distinguished, an additional total value of $ \frac{4}{3} \cdot d(v_0,v_{i-1}) - O(1)$ must be assigned to the vertices $u$ with $p(u) < i$.
Thus, we have 
\begin{align*}
\bdim(H_k-e) - \bdim(H_k)&\geq f(z)-g(z_1')-g(z_2') + \frac{4}{3}\cdot  \frac{3k}{2}   - O(1)\\
&\geq 4.75k - 6k + 2k - O(1)\\
&= \frac{k}{2} +\Omega(k),
\end{align*}
as desired.
\end{proof}

While the value $\bdim(G-e) - \bdim(G)$ can be arbitrarily large, the ratio $\frac{\bdim(G-e)}{\bdim(G)}$ is bounded. We prove this below, using some ideas from the proof that $\dim(G-e)\leq \dim(G)+2$ in \cite{eroh2015effect}. Recall that a \textit{geodesic} is a shortest path between two points.

\begin{reptheorem}{thm: last}
For all graphs $G$ and any edge $e\in E(G)$, we have $\frac{\bdim(G-e)}{\bdim(G)} \leq 3$.
\end{reptheorem}
\begin{proof}
Let $f$ be an efficient resolving broadcast of $G$, and let vertices $u$ and $v$ be the endpoints of edge $e$. Let $b = \max_{v\in V(G)} f(v)$.
We will show that function $f'$, which is identical to $f$, except with $f'(u) = f'(v) = b$, is a resolving broadcast of $G-e$. Then, we will be done since $$3\bdim(G) = 3\sum_{w\in V(G)} f(w) \geq \sum_{w\in V(G-e)}f'(w)\geq \bdim(G-e).$$

Let $z\in \supp_G(f)$, and let $x$ and $y$ be two vertices
with $d_{f(z)}(x,z) \neq d_{f(z)}(y,z)$ in graph $G$. Suppose that $x$ and $y$ are no longer resolved by $z$ after the edge $e$ is deleted; that is, $d_{f(z)}(x,z) = d_{f(z)}(y,z)$ in graph $G-e$. Then, we must have $d_G(u,z) \neq d_G(v,z)$ since removing edge $e=uv$ increases the distance from $z$ to at least one vertex in the graph.
Without loss of generality, we assume that $d_G(v,z) < d_G(u,z)$.

We consider two cases and show that $u$ resolves $x$ and $y$ in graph $G-e$ in both cases; that is, we show that we have $d_{f'(u)}(x,u) \neq d_{f'(u)}(y,u)$ in graph $G-e$ in both cases.

\noindent \textbf{Case 1.} Removing edge $e$ only increases the distance from $z$ to one of $x$ and $y$ (say $x$). \\
Edge $e$ must lie on every $x-z$ geodesic in $G$. 
Since $d_G(v,z) < d_G(u,z)$, we have an $x-u$ geodesic in $G$ that does not go through edge $e$. 
Moreover, we have 
\begin{align*}
f'(u) &\geq f(z)\geq \min\left\{d_G(x,z),d_G(y,z)\right\} = d_G(x,z)\geq d_G(x,u) = d_{G-e}(x,u).
\end{align*}
The above inequality shows that $u$ reaches $x$ with respect to $f'$ in graph $G-e$. 
Thus, it remains to be shown that $d_{G-e}(x,u) \neq d_{G-e}(y,u)$ in this case.

\noindent \textit{Subcase 1.} $f(z)\geq d_G(y,z) = d_{G-e}(y,z)$.\\
In this subcase, $d_{f(z)}(x,z) = d_{f(z)}(y,z)$ in graph $G-e$ implies that $d_{G-e}(x,z) = d_{G-e}(y,z)$, and so
\begin{align*}
d_{G-e}(x,u) &= d_G(x,u) = d_G(x,z) - d_G(z,u) < d_{G-e}(x,z) - d_G(z,u)\\
&= d_{G-e}(y,z) - d_G(z,u) = d_{G}(y,z) - d_G(z,u)\leq d_G(y,u) \\
&\leq d_{G-e}(y,u).
\end{align*}
\noindent \textit{Subcase 2.} $f(z) < d_G(y,z) = d_{G-e}(y,z)$.\\
If $d_{G-e}(x,u) = d_{G-e}(y,u)$, then we have
\begin{align*}
d_G(y,z) &\leq d_{G}(y,u) + d_G(u,z) \leq d_{G-e}(y,u) + d_G(u,z)\\
&= d_{G-e}(x,u) + d_G(u,z) = d_{G}(x,u) + d_G(u,z)\\
&= d_G(x,z) \leq f(z),
\end{align*}
a contradiction.

\noindent \textbf{Case 2.} Removing edge $e$ increases the distance from $z$ to both $x$ and $y$.\\
Edge $e$ must lie on every $x-z$ geodesic and every $y-z$ geodesic in graph $G$. Since $d_G(v,z) < d_G(u,z)$, we have $d_G(u,x) < d_G(v,x)$ and $d_G(u,y) < d_G(v,y)$. Because $z$ resolves $x$ and $y$ in $G$, at least one of $x$ and $y$ (say $x$) is reached by $z$ in $G$.
Then, $$f'(u) \geq f(z) \geq d_{G}(x, z) \geq d_G(x,u) = d_{G-e}(x,u)$$ and $$d_{G-e}(x,u) = d_G(x,u)  \neq d_G(y,u)=d_{G-e}(y,u),$$ so vertex $u$ resolves vertices $x$ and $y$ in graph $G-e$.
\end{proof}

\section{Future Work}
\label{Section: Future_Work}

In \cref{corollary: adim_of_acyclic}, we showed that $\adim(G) = \Omega(\sqrt{n})$ for all acyclic graphs $G$ of order $n$. To our knowledge, the best such lower bound before our work is the $\Omega(\log{n})$ bound on the adjacency dimension of general graphs of order $n$ given by Geneson and Yi in \cref{thm: lowerbound_adim_bdim}, which they showed to be asymptotically optimal using a family of graphs constructed by Zubrilina in \cite{zubrilina2018edge}. We ask if our lower bound on the adjacency dimension of acyclic graphs is asymptotically optimal.
\begin{question}
Is there a family of acyclic graphs $\set{G_k}_{k\in \mathds{Z}^+}$ with $\adim(G_k) = \Theta\paren{\sqrt{|V(G_k)|}}$ for every $k\in \mathds{Z}^+$?
\end{question}

The bounds that we derived in \cref{thm: acyclic_lower_bound} and \cref{thm: last} are sharp up to a constant factor. Sharper bounds may be obtained by examining the steps of the proofs more carefully.
Additionally, it would be interesting to determine the exact broadcast dimension of some special graphs for which the broadcast dimension is currently only known up to a constant factor.

\begin{question}
What is the broadcast dimension of the grid graph $P_m \square P_n$?
\end{question}
\begin{question}
What is the broadcast dimension of the graph $F_k$ from \cref{def: F_k}?
\end{question}

We note that the broadcast dimension of the grid graph $P_m \square P_n$ is at most $2m+2n$: for paths $P_m : x_1, x_2,\dots x_m$ and $P_n : y_1, y_2,\dots y_n$, the function $f$ that assigns $m+n$ to $\paren{x_{1}, y_1}$ and $\paren{x_{1}, y_{n}}$ and assigns 0 to the rest of the vertices is a resolving broadcast of $P_m \square P_n$.
Additionally, the broadcast dimension of $F_k$ is at most $3k$: function $f$ with $f(v_0) = 2k$, $f(v_k) = k$, and $f(w)=0$ for all $w\in V(F_k) - \{v_0, v_k\}$ is a resolving broadcast of $F_k$. \cref{lemma: F_k} makes partial progress towards finding the broadcast dimension of $F_k$.

In \cref{Section: Edge_Deletion}, we show that both $\bdim(G-e)-\bdim(G)$ and $\bdim(G) - \bdim(G-e)$ can be arbitrarily large and that $\frac{\bdim(G-e)}{\bdim(G)} \leq 3$ for all graphs $G$ and any edge $e\in E(G)$. These results naturally lead us to ask the following question:

\begin{question}
Is $\frac{\bdim(G)}{\bdim(G-e)}$ bounded from above for all graphs $G$ and any edge $e\in E(G)$?
\end{question}

On a similar note, Geneson and Yi showed in \cite{geneson2020broadcast} that both $\frac{\bdim(G)}{\bdim(G-v)}$ and $\bdim(G-v) - \bdim(G)$ can be arbitrarily large. 
The corresponding problem for $\frac{\bdim(G-v)}{\bdim(G)}$ remains open. 

\begin{question}
Is $\frac{\bdim(G-v)}{\bdim(G)}$ bounded from above for all graphs $G$ and any vertex $v\in V(G)$?
\end{question}

To better understand how metric dimension and broadcast dimension compare to each other, it would be interesting to derive more properties of broadcast dimension that are analogues to known properties of metric dimension. For example:

\begin{question}
For a graph $G$ and $n\in  \mathds{Z}^+$, bound $\bdim(G\square P_n)$ and $\bdim(G\square C_n)$ in terms of some function of $G$ and $n$.
\end{question}

\begin{question}
For graphs $G$ and $H$, bound $\bdim(G\square H)$ in terms of some function of $G$ and $H$.
\end{question}

\begin{question}
Is determining the broadcast dimension of a graph an NP-hard problem? 
\end{question}

It is NP-hard to determine the metric dimension and adjacency dimension of a general graph (see \cite{garey1979computers}, \cite{fernau2018adjacency}, respectively). Determining the domination number of a general graph is also an NP-hard problem \cite{garey1979computers}. Heggernes and Lokshtanov \cite{heggernes2006optimal} found a polynomial-time algorithm for computing the broadcast domination number of arbitrary graphs, and both the domination number and broadcast domination number of a tree can be determined in linear time (see \cite{cockayne1975linear},\cite{dabney2009linear}, respectively). We ask the corresponding question for the broadcast dimension of trees.

\begin{question}
Is there a polynomial-time algorithm for determining the value of $\bdim(T)$ for every tree $T$?
\end{question}

We refer to \cite{geneson2020broadcast} for more open questions about broadcast dimension. Finally, we note that it would also be interesting to study the broadcast dimension of directed graphs and graphs with weighted edges.

\section{Acknowledgments}
This research was conducted at the 2020 University of Minnesota Duluth Research Experience for Undergraduates (REU) program, which is supported by NSF-DMS grant 1949884 and NSA grant H98230-20-1-0009.  I would like to thank Joe Gallian for organizing the program, suggesting the problem, and supervising the research. I would also like to thank Amanda Burcroff, Brice Huang, and Joe Gallian for reading this paper and giving valuable suggestions and Amanda Burcroff for useful discussions over the course of the program.


\textsc{Massachusetts Institute of Technology, Cambridge, MA 02139, USA}

\emph{E-mail address: }\href{mailto:eyzhang@mit.edu}{\tt eyzhang@mit.edu}

\end{document}